\documentclass[11pt, reqno]{amsart}
\usepackage{amsmath,amsfonts,amssymb,amsthm}
\usepackage{epsfig}

\usepackage{bbm}
\usepackage{enumerate}
\usepackage{amstext}
\usepackage{xspace}
\usepackage{amsfonts}
\usepackage{amsmath}
\usepackage{amssymb}
\usepackage{amstext}
\usepackage{amsthm}       
\usepackage{xspace}
\usepackage[active]{srcltx}

\newcommand{\I}{\mathcal I}
\newcommand{\ind}{\{1,\dots,m\}}

\newcommand{\F}{\mathcal F}
\newcommand{\R}{\mathbb R}

\newcommand{\N}{\mathbb N}

\newcommand{\T}{\mathbb T}
\newcommand{\Z}{\mathbb Z}

\newcommand{\eps}{\varepsilon}
\newcommand{\ucv}{\rightrightarrows}
\newcommand{\uu}{\mathbf u}
\newcommand{\vv}{\mathbf v}
\newcommand{\ww}{\mathbf w}
\newcommand{\aaa}{\mathbf a}

\newcommand{\1}{\mathbbm1}

\newcommand{\A}{\mathcal{A}}
\newcommand{\hh}{\mathcal{H}}

\newcommand{\llambda}{\boldsymbol\lambda}
\newcommand{\mmu}{\boldsymbol\mu}

\renewcommand{\S}{\mathcal{S}}

\newcommand{\tagliato}{$\kern-5 mm -$}
\newcommand{\tagliat}{$\kern-4 mm -$}

\newcommand{\D}[1]{\mbox{\rm #1}}
\newcommand{\dd}{\D{d}}

\newtheorem{teorema}{Theorem}[section]
\newtheorem{prop}[teorema]{Proposition}
\newtheorem{lemma}[teorema]{Lemma}
\newtheorem{definition}[teorema]{Definition}
\newtheorem{cor}[teorema]{Corollary}

\newtheorem{guess}[teorema]{Remark}
\newtheorem{example}[teorema]{Example}

\newenvironment{dimo}{{\bf\noindent Proof.}}{\qed}
\newenvironment{oss}{\begin{guess} \begin{rm}}{\end{rm} \end{guess}}

\begin{document}

\title{Aubry sets for weakly coupled systems \\ of Hamilton--Jacobi equations}
\author{Andrea Davini \and Maxime Zavidovique}
\address{Dip. di Matematica, {Sapienza} Universit\`a di Roma,
P.le Aldo Moro 2, 00185 Roma, Italy}
\email{davini@mat.uniroma1.it}
\address{
IMJ (projet Analyse Alg\' ebrique), UPMC,  
4, place Jussieu, Case 247, 75252 Paris C\' edex 5, France}
\email{zavidovique@math.jussieu.fr} 

\begin{abstract}
We introduce a notion of Aubry set for weakly coupled systems of Hamilton--Jacobi equations on the torus and characterize it as the region where the obstruction to the existence of globally strict critical subsolutions concentrates. As in the case of a single equation, we prove the existence of critical subsolutions which are strict and smooth outside the Aubry set. This allows us to derive in a simple way a comparison result among critical sub and supersolutions with respect to their boundary data on the Aubry set, showing in particular that the latter is a uniqueness set for the critical system. 
We also highlight some rigidity phenomena taking place on the Aubry set. 
\end{abstract}
\maketitle
\section*{Introduction}

In this paper we will consider a {\em weakly coupled system} of Hamilton--Jacobi equations of the form  
\begin{equation}\label{intro wcs}
H_i(x,Du_i)+\sum_{j=1}^m b_{ij}(x)u_j(x)=a\quad\hbox{in $\T^N$}\qquad\hbox{for every $i\in\ind$,}
\end{equation}
where $a$ is a real constant, $H_1,\dots,H_m$ are continuous Hamiltonians defined on the cotangent bundle of $\T^N$, convex and coercive in the momentum variable, and $B(x):=\big(b_{ij}(x)\big)$ is a continuous $m\times m$ matrix satisfying
\[
b_{ij}(x)\leqslant 0\quad \hbox{for $j\not=i$,}\quad \sum_{j=1}^m b_{ij}(x)=0\quad
\qquad\hbox{for every $x\in\T^N$ and $i\in\{1,\dots,m\}$.} 
\]
The coupling matrix will be also assumed {\em irreducible}, meaning, roughly speaking, that the coupling is non--trivial and the system cannot be split into independent subsystems. Under these assumptions it is known that there is a unique value for $a$,  denoted by $c$ and termed {\em critical} in the sequel, such that the corresponding system \eqref{intro wcs} admits viscosity solutions. 
Critical solutions are, instead, not unique: for instance, it is easily seen that $\uu+(\lambda,\dots,\lambda)$ is a solution if $\uu:=(u_1,\dots,u_m)$ is so, for every $\lambda\in\R$; but in general they are not unique even up to addition of vectors of the form $(\lambda,\dots,\lambda)$.

In this paper, we perform a qualitative analysis on the critical weakly coupled system and we investigate  the kind of non--uniqueness phenomena taking place at the critical level. This is obtained by extending to the case at issue some PDE aspects of the so called weak KAM Theory, holding when there is one equation only and with no coupling.  
The core of our analysis is discovering that, at the critical level, a very rigid object appears, characterized as the region where the obstructions of strict critical subsolutions concentrates and named after {\em Aubry} in analogy with the case of a single equation. We explore the properties of this object, in particular we show that it is a uniqueness set for the critical system and we highlight some rigidity phenomena taking place on it. 

Weakly coupled systems of the kind herein considered arise naturally in optimal control problems associated with  randomly switching costs, where the switching is governed by specific Markov chains, see \cite{FlemSon, YinZhang}. In the PDE literature, they have been studied  as a particular instance of monotone systems, see \cite{Engler, IsKo, Len88},  and, more recently, in a series of papers addressed to study related homogenization problems in the periodic \cite{CamLey, MT2} and  stationary ergodic setting \cite{Fe13}, and the long time behaviour of solutions to the associated evolutionary version of the system  \cite{leyetal, mitaketran, MT3, Ng}. 
These works can be read as a generalization of results and techniques established in the case of a single equation: see \cite{LPV} for the basic homogenization results of Hamilton--Jacobi equations in the periodic setting,  \cite{RezTa00, Sou99, ArSo13} for the subsequent extension to the stationary ergodic setting, \cite{NR,Fa1,BS, Ro01, DS, Ishii08} for the long time convergence towards stationary solutions.  

In the afore mentioned references the existence of a unique constant for which the system \eqref{intro wcs} is solvable is obtained via  arguments analogous to the ones introduced in \cite{LPV}, in particular by making use of the so called {\em ergodic approximation} (and for this reason, the constant $c$ is also termed {\em ergodic}). A different approach, based on a second order approximation and on the new adjoint methods introduced by Evans \cite{Ev10} for Hamilton--Jacobi equations, is instead proposed in \cite{CaGoTran}. On the other hand, little attention has been devoted  to the qualitative study of the associated critical, or ergodic, system and to the kind of nonuniqueness phenomena that occur. Some results in this direction have been obtained in \cite{leyetal}, where the existence of a uniqueness set for the critical system is derived in the specific case{\footnote{The conditions on the coupling matrix are instead more general, in the sense that the sums of the elements on each 
row of $B(x)$ are 
not required to vanish identically on the whole torus, but only on a part of it.}} when the Hamiltonians are of the form 
\[
 H_i(x,p):=F_i(x,p)-V_i(x)\qquad\hbox{for every $(x,p)\in\T^N\times\R^N$ and $i\in\ind$,}
\]
with $V_i$ non--negative functions satisfying $\bigcap_{i=1}^m V_i^{-1}(\{0\})\not=\varnothing$, and $F_i$ convex and coercive in $p$, and such that
$$
F_i(x,p)\geqslant F_i(x,0)=0\quad\hbox{for every $(x,p)\in\T^N\times\R^N.$}
$$  

For the analysis performed in this paper we employ techniques borrowed from the weak KAM Theory. This theory, developed by Fathi to study the dynamics of Hamiltonian systems, has revealed to be a powerful tool for a fine qualitative analysis of the stationary critical equation associated with a Tonelli Hamiltonian. Fathi discovered in fact \cite{Fa3,Fa4, Fa5, Fa1} that the Aubry set
appears as the dynamical part of the set of points where all the critical viscosity subsolutions are smooth;
moreover, it is a uniqueness set for the critical equation. His approach relies on a characterization of the critical solutions in terms of variational formulae. We refer the reader to  \cite{Fa} for an expository presentation of the theory for Tonelli Hamiltonians and to \cite{Fa12} for the extension of its PDE aspects to less regular cases.

The generalization of the weak KAM Theory proposed in this paper was however nontrivial. A first serious difficulty is that we could not find suitable variational formulae to express the solutions of the system. We have overcome this problem by choosing a more PDE oriented viewpoint to weak KAM Theory and by making use of viscosity solution techniques. A second issue is that, to extend this approach to the case at issue, we needed to better understand the different structure of the problem and the role of the coupling. This has been achieved by recognizing and exploiting the vectorial nature of the system.

The starting point of our study relies on a different definition of the critical value $c$, given as the minimal $a\in\R$ for which  the corresponding weakly coupled system admits viscosity subsolutions. We subsequently show that the associated critical weakly coupled system is the unique among those of the family \eqref{intro wcs} that admits viscosity solutions. This characterization of the critical value is not new, see for instance \cite{mitaketran}, but our approach to the existence part is different: it is based on a fixed point argument in the spirit of \cite{Fa} and on the properties of the semigroup associated with the evolutionary version of the system \eqref{intro wcs}.

We then study the corresponding {critical weakly coupled system} and show that the obstruction to the existence of globally strict subsolutions  is not spread indistinctly on the torus, but concentrates on a closed set $\A$, that we call {\em Aubry set} in analogy to the case of a single equation. Moreover, we show the existence of {critical subsolutions} smooth and strict outside the Aubry set. This allows us to derive in a simple way a comparison result among critical sub and supersolutions satisfying suitable ``boundary'' conditions on $\A$, see Theorem \ref{teo comparison}. This generalizes to our setting Theorem 3.3 in \cite{leyetal}. In particular, we infer that the Aubry set is a uniqueness set for the critical system. 
We furthermore show that the trace of any critical subsolution on $\A$ can be extended on the whole torus in such a way that the output is a critical solution, see Theorem \ref{teo Lax-type}.  

Our study highlights some rigidity phenomena taking place on the Aubry set. First, we show that any pair of critical subsolutions differ, at each point $y$ of $\A$, by a vector of the form $k\,(1,1,\dots,1) $, see Theorem \ref{rigid}. 
In the particular case when there exists a critical subsolution of the kind $\big(v(x),v(x),\dots,v(x)\big) $, we infer that any other critical subsolution is of this form on $\A$. This accounts for the kind of symmetries already observed in \cite{leyetal} for the particular class of Hamiltonians therein considered, see Section \ref{sez ley} for more details. 

A second rigidity phenomenon that we point out is when the Hamiltonians are additionally assumed strictly convex in the momentum: in this case we prove that, at any point of the Aubry set, the intersection of the reachable gradients of all the critical subsolutions is always nonempty, see Proposition \ref{prop common gradient}. This can be regarded as a weak version  of a result holding in the scalar case: under suitable regularity assumptions on the Hamiltonian, it is in fact well known that critical  subsolutions are all differentiable on the Aubry set and have the same gradient, see \cite{Fa, FSC1, FS05}. 

We end our study by presenting some situations where more explicit results can be obtained for the critical value and for the Aubry set, see Section \ref{examples}. In the first example, we focus on the setting studied in \cite{leyetal} and we show that our notion of Aubry set is consistent with the one therein given. Next we show how such analysis can be extended to a variant of this model  considered in \cite{Ng}. For this, we make use of a general result on weakly coupled systems with constant coupling matrices which is of independent interest, see Proposition \ref{prop new}. 
The last example contains, as a particular instance, the case when the Hamiltonians are all equal, say to $H$. In this specific situation, we show that the critical value and the Aubry set of the weakly coupled system agree with the critical value and the Aubry set of $H$. We furthermore show that  
the solutions of the critical system are all of the  form $u(x)(1,\dots,1)$, where $u$ is a critical solution for $H$, thus showing that, as far as critical solutions of the system are concerned, the coupling is not playing any effective role.\smallskip   

This paper is organized as follows. In Section \ref{sez notation} we fix the notations and assumptions, and we give a brief overview of existing results on weakly coupled systems.
Section \ref{Critical v} is devoted to the definition of the critical value and to the study of its main properties.  
Some proofs are postponed to the Appendix \ref{appendix semiconcave}. 
In Section \ref{aubry mane} we give the definition of Aubry set and explore its properties. 
The first part of Section \ref{regularization} is devoted to the regularization of subsolutions outside of the Aubry set, while in the second part we prove the rigidity phenomenon enjoyed by the reachable gradients of the critical subsolutions previously described. The other rigidity phenomenon is instead proved at the beginning of Section \ref{comparison} and some consequences are drawn. In the remainder of the section we prove the comparison principle, Theorem \ref{teo comparison}, and we show how the trace of a critical subsolution on the Aubry set can be extended to the whole torus to produce a critical solution, see Theorem \ref{teo Lax-type}. In Section \ref{examples} we illustrate our theory on some examples.\medskip

\smallskip\indent{\textsc{Acknowledgements. $-$}}
This research was initiated while the first author was visiting the Institut de Math\'ematiques de Jussieu,  
Universit\'e Pierre et Marie Curie, Paris, that he acknowledges for the kind hospitality.  
The first author has been supported by {\em Sapienza} Universit\`a di Roma through the Research Project 2011 {\em Analisi ed approssimazione di modelli differenziali nonlineari in fluidodinamica e scienza dei materiali}. 
The second author is partially supported by ANR-12-BLAN-WKBHJ.
\bigskip
 
 \numberwithin{equation}{section}

\begin{section}{Preliminaries}

\begin{subsection}{Notations}\label{sez notation}

Throughout the paper, we will denote by $\T^N = \R^N / \Z^N$ the $N$--dimensional flat torus, where $N$ is
an integer number. The scalar product in $\R^N$ will be denoted by
$\langle\,\cdot\;, \cdot\,\rangle$, while the symbol $|\cdot|$
stands for the Euclidean norm. Note that the latter induces a distance
on $\T^N$,  denoted by $d(\cdot,\cdot )$, defined as
\[
d(x,y):=\min_{\kappa\in\Z^N}|x-y+\kappa |\qquad\hbox{for every $x,y\in\T^N$}.
\]
More generally, if $X\subset \T^N$ is a set, we will denote by $d(x,X):=\inf\limits_{y\in X} d(x,y)$ the distance from a  point $x\in \T^N$ to $X$.
We will denote by $B_R(x_0)$ and $B_R$ the closed balls in $\T^N$ of
radius $R$ centered at $x_0$ and $0$, respectively.

With the symbol $\R_+$ we will refer to the set of
nonnegative  real numbers.  We say
that a property holds {\em almost everywhere} ($a.e.$ for short)
in a subset $E$ of  $\T^N$ if it holds up to a {\em negligible} subset of $E$, i.e. a
subset of zero $N$--dimensional Lebesgue measure.

\indent By modulus we mean a nondecreasing function from $\R_+$ to
$\R_+$, vanishing and continuous at $0$. A function $g:\R_+\to\R$
will be termed {\em coercive} if
$\displaystyle{\lim_{h\to +\infty} {g(h)}=+\infty}$.

We will say that $(\rho_n)_n$ is a sequence of {\em standard mollifiers} if \ $\rho_n(x):=n^N\rho(nx)$ in $\R^N$\quad for each $n\in\N$, where $\rho$ is a smooth, non--negative function on $\R^N$, supported in $B_1$ and such that its integral over $\R^N$ is equal to $1$.\smallskip\par  

Given a continuous function $u$ on $\T^N$, we will call  {\em subtangent} (respectively, {\em supertangent}) of $u$ at $x_0$ a function $\phi$ of class $ C^1$ in a neighborhood $U$ of $x_{0}$ such that $u-\phi$ has a local minimum (resp., maximum) at $x_{0}$. Its gradient $D\phi(x_0)$ will be called a {\em subdifferential}  (resp. {\em superdifferential}) of $u$ at $x_0$. The set of sub and superdifferentials of $u$ at $x_0$ will be denoted $D^-u(x_0)$ and 
$D^+u(x_0)$, respectively.  The function $\phi$ will be furthermore termed {\em strict subtangent} (resp., {\em strict supertangent}) if $u-\phi$ has a {\em strict} local minimum (resp., maximum) at $x_{0}$. Any subtangent (resp., supertangent) $\phi$ of $u$ can be always assumed strict at $x_{0}$ without affecting $D\phi(x_0)$ by possibly replacing it with 
$\phi-d^2(x_{0},\cdot)$ (resp. $\phi+d^2(x_{0},\cdot)$).  We recall that $u$ is
differentiable at $x_0$ if and only if $D^+u(x_0)$ and $D^-u(x_0)$ are
both nonempty. In this instance, $D^+u(x_0)=D^-u(x_0)=\{Du(x_0)\}$. We refer the
reader to \cite{CaSi00} for the 
proofs.

When $u$ is locally Lipschitz in $\T^N$, we will denote by $\partial^*u(x_0)$ the set of 
{\em reachable gradients} of $u$ at $x_0$, that is the set
\[
\partial^* u(x_0)=\{\lim_n D u(x_n)\,:\,\hbox{$u$ is differentiable at $x_n$, $x_n\to x_0$}\,\},
\]
while the {\em Clarke's generalized gradient} $ \partial^c  u(x_0)$ is the closed convex hull of $\partial^* u(x_0)$. 
The set $\partial^c u(x_0)$ contains both $D^+u(x_0)$ and $D^-u(x_0)$, in particular $Du(x_0)\in \partial^c u(x_0)$ at any differentiability point $x_0$ of $u$. We refer the reader to \cite{Cl} for a detailed treatment of the subject.

\indent 
We will denote by $\|g\|_\infty$ the usual $L^\infty$--norm of $g$, where the latter is a   
measurable real function defined on $\T^N$. 
We will write
$g_n\ucv g$ in $\T^N$ to mean that the sequence of
functions $(g_n)_n$ uniformly converges to $g$ in $\T^N$, i.e. $\|g_n-g\|_{\infty}\to 0$. 
We will denote by $\left(\D{C}(\T^N)\right)^m$ the Banach space of continuous functions $\uu=(u_1,\dots,u_m)^T$ 
from $\T^N$ to $\R^m$, endowed with the norm 
\[
 \|\uu-\vv\|_\infty=\max_{1\leqslant i\leqslant m}\|u_i-v_i\|_\infty,\qquad\hbox{$\uu,\vv\in\left(\D{C}(\T^N)\right)^m$}.
\]
We will write $\uu^n\ucv \uu$ in $\T^N$ to mean that $\|\uu^n-\uu\|_\infty\to 0$. A function $\uu\in \left(\D{C}(\T^N)\right)^m$ will be termed Lipschitz continuous if each of its components is $\kappa$--Lipschitz continuous, for some $\kappa>0$. Such a constant $\kappa$ will be called a {\em Lipschitz constant} for $\uu$. The space of all such functions will be denoted by $\left(\D{Lip}(\T^N)\right)^m$. 
\smallskip\par

We will denote by $\1=(1,\cdots,1)^T$ the vector of $\R^m$ having all components equal to 1, where the upper--script symbol $T$ stands for the transpose. We consider the following partial relations between elements  $\mathbf{a},\mathbf{b}\in\R^m$:
$\aaa \leqslant \mathbf{b}$ (respectively,  
$\aaa <\mathbf{b}$)   if $a_i\leqslant b_i$ (resp., $<$) for every $i\in\ind$. Given 
two functions $\uu,\vv:\T^N\to\R^m$, we will write $\uu \leqslant \vv$  in $\T^N$ (respectively, $<$) to mean that $\uu(x)\leqslant \vv(x)$ \big(resp., $\uu(x)<\vv(x)$\big)  {for every $x\in\T^N$}. 
\smallskip\par

\end{subsection}

\begin{subsection}{Linear algebra}\label{sez linear algebra}
Here we briefly present some elementary linear algebraic results concerning coupling matrices. 

\begin{definition}
Let $B=(b_{ij})_{i,j}$ be a $m\times m$--matrix.
\begin{itemize}
\item[{\em (i)}]
We say that $B$ is a {\em  coupling matrix} if it satisfies the following conditions:
\begin{equation}\label{hyp coupling}
b_{ij}\leqslant 0\ \hbox{for $j\not=i$,}\quad\quad \sum_{j=1}^m b_{ij}\geqslant 0
\qquad\hbox{for any $i\in\{1,\dots,m\}$.}\tag{C}
\end{equation}
It is additionally termed {\em degenerate} if
\[
\sum_{j=1}^m b_{ij}=0
\qquad\hbox{for any $i=1,\dots,m$.}
\]
\item[{\em (ii)}]
We say that $B$ is {\em irreducible} if for every subset $\I\subsetneq\{1,\dots,m\}$ there exist $i\in\I$ and $j\not\in\I$ such that $b_{ij}\not=0$.\\
\end{itemize}
\end{definition}

When a coupling matrix is also irreducible, further information can be derived on its elements. We have 

\begin{prop}\label{prop positive diagonal}
Let $B=(b_{ij})_{i,j}$ be an irreducible  $m\times m$  coupling matrix. Then \quad $b_{ii}>0$ \quad for every $i\in\ind$.
\end{prop}

\begin{dimo}
Indeed, if $b_{i_0 i_0}=0$ for some $i_0\in\ind$, condition \eqref{hyp coupling} would imply $b_{i_0 j}=0$ for every $j\in\ind$, in  contradiction with the fact that $B$ is irreducible.
\end{dimo}

The following invertibility  criterion holds:

\begin{prop}\label{prop Ker}
Let $B=(b_{ij})_{i,j}$ be an $m\times m$ irreducible  coupling matrix. Then
\begin{itemize}
 \item[{\em (i)}]\quad $\D{Ker}(B)\subseteq\D{span}\{(1,\dots,1)^T\}=\R \1$;\smallskip
\item[{\em (ii)}]\quad $\D{Ker}(B)=\D{span}\{(1,\dots,1)^T\} = \R \1$ \quad if and only if \quad  $B$ is degenerate.\medskip
\end{itemize}
In particular, $B$ is invertible if and only if
\[
 \sum_{j=1}^m b_{ij}> 0\qquad\hbox{for some $i\in\{1,\dots,m\}$.}
\]
\end{prop}

\begin{dimo}
We first remark that, by assumption \eqref{hyp coupling},
\begin{equation}
 b_{ii}\geqslant \sum_{j\not=i} |b_{ij}|\qquad\hbox{for every $i\in\{1,\dots,m\}$.}
\end{equation}
Let us prove {\em (i)}. Let $\vv=(v_1,\dots,v_m)^T\in\D{Ker}(B)$ and set
\[
 \I=\big\{i\in\{1,\dots,m\}\,:\,v_i=\max\{v_1,\dots,v_m\}\,\big\}.
\]
We claim that $\I=\{1,\dots,m\}$. Indeed, if this were not the case, by the irreducible character of $B$ there would exist $i\in\I$ and $k\not\in\I$ such that $b_{ik}\not=0$. Since $B\vv=0$, we would get in particular
\[
 b_{ii}v_i=\sum_{j\not=i} v_j |b_{ij}|\leqslant v_i \sum_{j\not=i} |b_{ij}|\leqslant v_i b_{ii}.
\]
Then the inequalities must be  equalities. We infer 
\[
 v_j|b_{ij}|=v_i|b_{ij}|\qquad\hbox{for every $j\not=i$,}
\]
in particular $v_k=v_i=\max\{v_1,\dots,v_m\}$, yielding that $k$ belongs to $\I$, a contradiction.

The remainder of the statement trivially follows from item {\em (i)}.
\end{dimo}
\bigskip

%
%
%

The following proposition gives an obstruction to being in the image of a degenerate coupling matrix.
\begin{prop}\label{prop Im}
Let $B$ a coupling and degenerate $m\times m$ matrix. If $ \aaa =(a _1,\dots,a _m)^T$  satisfies $a _i>0$ for every $i\in\ind$, then $ {\aaa }\not\in  \D{Im}(B)$.
\end{prop}

\begin{dimo}
Let us assume by contradiction that there exists ${  \vv}=(v_1,\dots,v_m)^T$ such that
\[
 B  \vv= {\aaa}.
\]
Let $v_k=\min\{v_1,\dots,v_m\}$. We have
$$a_k=\sum_{j=1}^m b_{kj}v_j\leqslant \sum_{j=1}^m b_{kj}v_k =0,$$
in contradiction with the hypothesis $a _k>0$.\medskip
\end{dimo}
\end{subsection}

\begin{subsection}{Weakly coupled systems}\label{Wcs}
Throughout the paper, we will call {\em convex Hamiltonian} a
function $H$ satisfying the following set of assuptions:\smallskip
\begin{itemize}
    \item[(H1)] \quad $H:\T^N\times\R^N\to\R\qquad$ is
    continuous;\smallskip\\
    \item[(H2)] \quad $p\mapsto H(x,p)\qquad\hbox{is  convex on $\R^N$ for any
    $x\in \T^N$;}$ \smallskip\\
     \item[(H3)] \quad there
     exist two  coercive  functions $\alpha,\beta:\R_+\to\R$ such
     that
     \[
     \alpha\left(|p|\right)\leqslant H(x,p)\leqslant \beta\left(|p|\right)\qquad\hbox{for all
     $(x,p)\in \T^N\times\R^N$.}
     \]
\end{itemize}
The Hamiltonian $H$ will be termed {\em strictly convex} if it additionally satisfies the following stronger 
assumption:
\begin{itemize}
\item[(H2)$'$] \quad $p\mapsto H(x,p)\qquad\hbox{is strictly
convex on $\R^N$ for any $x\in \T^N$}$.\medskip
\end{itemize}

Moreover, we will denote by $B(x)=\big(b_{ij}(x)\big)_{i,j}$ an $m\times m$--matrix with continuous coefficients $b_{ij}(x)$ on $\T^N$. If not otherwise stated, the following hypotheses will be always assumed: 
\begin{itemize}
 \item[(B1)]\quad $B(x)$ is an irreducible coupling matrix for every $x\in\T^N$;\medskip
 \item[(B2)]\quad $B(x)$ is degenerate for every $x\in\T^N$.\medskip
\end{itemize}

Let $H_1(x,p),\dots,H_m(x,p)$ be convex Hamiltonians, i.e. functions satisfying conditions (H1)--(H3).
We are interested in weakly coupled systems of the form
\begin{equation}\label{wcoupled system}
H_i(x,Du_i)+\big(B(x)\uu(x)\big)_i=a_i\quad\hbox{in $\T^N$}\qquad\hbox{for every $i\in\ind$,}
\end{equation}
for some constant vector $\aaa=(a_1,\dots ,a_m)^T$, where $\uu(x)=\big(u_1(x),\dots,u_m(x)\big)^T$ and $\big(B(x)\uu(x)\big)_i$ denotes the $i$--th component of the vector $B(x)\uu(x)$, i.e. 
\[
\big(B(x)\uu(x)\big)_i=\sum_{j=1}^m b_{ij}(x)u_j(x). 
\]
\begin{oss}\label{oss monotone system}
The weakly coupled system \eqref{wcoupled system} is a particular type of {\em monotone system}, i.e. a system of the form
\[
 G_i\big(x,u_1(x),\dots,u_m(x),Du_i\big)=0\qquad\hbox{in $\T^N$}\qquad\hbox{for every $i\in\ind$,}
\]
where suitable monotonicity conditions  with respect to the $u_j$--variables are assumed on the functions $G_i$, see \cite{CamLey,Engler,Is92,IsKo,Len88}. In the specific case considered in this paper, the conditions assumed on the coupling matrix imply, in particular, that each function $G_i$ is strictly increasing in $u_i$ and non--increasing in $u_j$ for every $j\not=i$. This kind of monotonicity will be exploited in many points of the paper.
\end{oss}

Let $\uu\in\left(\D{C}(\T^N)\right)^m$. We will say that $\uu$ is a {\em viscosity subsolution} of \eqref{wcoupled system} if the following inequality holds for every $(x,i)\in\T^N\times\ind$
\begin{equation*}
H_i(x,p)+\big(B(x)\uu(x)\big)_i\leqslant a_i\quad\hbox{for every $p\in D^+ u_i(x)$.}
\end{equation*}
We will say that $\uu$ is a {\em viscosity supersolution} of \eqref{wcoupled system} if the following inequality holds for every $(x,i)\in\T^N\times\ind$
\begin{equation*}
H_i(x,p)+\big(B(x)\uu(x)\big)_i\geqslant a_i\quad\hbox{for every $p\in D^- u_i(x)$.}
\end{equation*}
We will say that $\uu$ is a {\em viscosity solution} if it is both a sub and a supersolution. In the sequel, solutions, subsolutions and supersolutions will be always meant in the viscosity sense, hence the adjective {\em viscosity} will be  omitted. 


Due to the convexity of the Hamitonian $H_i$, the following equivalences hold:

\begin{prop}\label{prop subsol equivalent def}
Let $a\in\R$, $i\in\ind$ and $\uu\in\left(\D{Lip}(\T^N)\right)^m$. The following facts are equivalent:\medskip
\begin{itemize}
 \item[\em (i)]\quad $H_i(x,p)+\big(B(x)\uu(x)\big)_i\leqslant a\qquad\ \ \quad\hbox{for every $p\in D^+ u_i(x)$ and $x\in\T^N$;}$\medskip
 \item[\em (ii)]\quad $H_i(x,p)+\big(B(x)\uu(x)\big)_i\leqslant a\qquad\ \ \quad\hbox{for every $p\in D^- u_i(x)$ and $x\in\T^N$;}$\medskip
\item[\em (iii)]\quad $H_i(x,p)+\big(B(x)\uu(x)\big)_i\leqslant a\qquad\ \ \quad\hbox{for every $p\in \partial^c u_i(x)$ and $x\in\T^N$;}$\medskip
\item[\em (iv)]\quad $H_i\big(x,Du_i(x)\big)+\big(B(x)\uu(x)\big)_i\leqslant a\quad\hbox{for a.e. $x\in\T^N$.}$\medskip
\end{itemize}
\end{prop}

Next, we state a proposition that will be needed in the sequel, see also \cite{Engler,Len88,Is92,IsKo} for similar results. 

\begin{prop}\label{prop Engler}
Let $\F$ be a subset of $\left(\D{C}(\T^N)\right)^m$ and define the functions $\underline\uu,\,\overline\uu$ on $\T^N$ by setting:
\[
 \underline u_i(x)=\inf_{\uu\in\F} u_i(x),\quad \overline u_i(x)=\sup_{\uu\in\F} u_i(x)\qquad\hbox{for every $x\in\T^N$ and $i\in\ind$.}
\]
Assume that $\underline\uu$ and $\overline\uu$ belong to $\left(\D{C}(\T^N)\right)^m$ and let $\aaa\in\R^m$. Then:\medskip
\begin{itemize}
 \item[\em (i)] if every $\uu\in\F$ is a subsolution of \eqref{wcoupled system}, then $\overline\uu$ is a subsolution of \eqref{wcoupled system};\medskip
\item[\em (ii)] if every $\uu\in\F$ is a supersolution of \eqref{wcoupled system}, then $\underline\uu$ is a supersolution of \eqref{wcoupled system}.\medskip
\end{itemize}
\end{prop} 

The previous proposition is analogous to a well known fact for scalar Hamilton--Jacobi equations, see for instance Section 2.6 in \cite{barles}. The proof can be easily recovered by arguing similarly and by exploiting the monotonicity of the system.\medskip\par


We will be also interested in the evolutionary counterpart of \eqref{wcoupled system}, i.e. the system 
\begin{equation}\label{evo wcoupled system}
\frac{\partial u_i}{\partial t}+H_i(x,D_xu_i)+\big(B(x)\uu(t,x)\big)_i=0\quad\hbox{in $(0,+\infty)\times\T^N$}\qquad\hbox{$\forall i\in\ind$,}
\end{equation}
where we have denoted by $\uu(t,x)=\big(u_1(t,x),\dots,u_m(t,x)\big)^T$.

The following comparison result holds, see for instance \cite{CamLey} for a proof.

\begin{prop}
Let $T>0$ and $\vv,\,\uu\in\big(\D{Lip}([0,T]\times\T^N)\big)^m$ be, respectively, a sub and a supersolution of \eqref{evo wcoupled system}. Then, for every $i\in\ind$, 
\[
 v_i(t,x)-u_i(t,x)\leqslant\max_{1\leqslant i \leqslant m}\,\max_{\T^N} \big(v_i(0,\cdot)-u_i(0,\cdot)\big),\qquad\hbox{$(t,x)\in [0,T]\times\T^N$}.
\]
\end{prop}

By making use of this proposition and of Perron's method, it is then easy to prove the following 

\begin{prop}\label{prop existence evo wcs}
Let $\uu_0\in \big(\D{Lip}(\T^N)\big)^m$. Then there exists a unique function $\uu(t,x)$ in $\big(\D{Lip}(\R_+\times\T^N)\big)^m$ that solves the system \eqref{evo wcoupled system} subject to the initial condition $\uu(0,x)=\uu_0(x)$ in $\T^N$. Moreover, the Lipschitz constant of $\uu(t,x)$ in $\R_+\times\T^N$ only depends on the Hamiltonians $H_1,\dots,H_m$ and on the Lipschitz constant of $\uu_0$.\medskip
\end{prop}

We will denote by $\S(t)\uu_0(x)$ the solution $\uu(t,x)$ of \eqref{evo wcoupled system} with initial datum $\uu_0$. This defines,  for every $t>0$, a map 
\[
\S(t):\big(\D{Lip}(\T^N)\big)^m\to \big(\D{Lip}(\T^N)\big)^m.
\]
We summarize in the next proposition the properties enjoyed by such maps, which come as an easy application of the above results.

\begin{prop}\label{prop semigroup}
For every $t,s>0$ and $\uu,\vv\in \big(\D{Lip}(\T^N)\big)^m$ we have:
\begin{itemize}
 \item[\em (i)]{\bf (Semigroup property)}\quad $\S(s)\big(\S(t)\uu\big)=\S(t+s)\uu$ \quad in $\T^N$;\medskip
 \item[\em (ii)]{\bf (Monotonicity)}\quad if \quad $\vv\leqslant\uu$ \quad in $\T^N$, then\quad  $\S(t)\vv\leqslant\S(t)\uu$\quad  in $\T^N$;\medskip
\item[\em (iii)]{\bf (Non--expansiveness property)}\quad $\displaystyle{\|\S(t)\vv-\S(t)\uu\|_\infty\leqslant\|\vv-\uu\|_\infty}$;\medskip
\item[\em (iv)]  for every $a\in\R$, \quad $\S(t)(\uu+a\1)=\S(t)\uu+a\1$\quad  in $\T^N$.\medskip
\end{itemize}
\end{prop}
  
The fact that the coupling matrix $B(x)$ is everywhere degenerate is crucial for assertion {\em (iv)}.\medskip 

\end{subsection}
\end{section}

\begin{section}{The critical value}\label{Critical v}

The purpose of this section is to  define the notion of critical value for weakly coupled systems and to prove some relevant properties of the corresponding critical system. 

We start by proving some {\em a priori} estimates for the subsolutions of a weakly coupled system of the form \eqref{wcoupled system}. The following notation will be assumed throughout the section:
\[
 \mu_i=\min_{(x,p)} H_i(x,p)\quad\hbox{for each $i\in\ind$,}\qquad \mu=\min_{i\in\ind}\mu_i.
\]

\begin{prop}\label{prop compactness}
Let $\aaa=(a_1,\dots,a_m)^T\in\R^m$ and $\uu \in \left(\D{C}(\T^N)\right)^m$ such that
\begin{equation}\label{eq compactness}
 \big(B(x)\uu (x)\big)_i\leqslant a_i\qquad\hbox{for every $x\in\T^N$ and $i\in\ind$.}
\end{equation}
Then there exists a constant $M_\aaa$ only depending on $\aaa$ and $B(x)$ such that
\begin{itemize}
 \item[\em (i)]\quad $\|u_i-u_j\|_\infty\leqslant M_\aaa\ \quad\qquad\hbox{for every $i,\,j\in\ind$;}$\medskip
\item[\em (ii)]\quad $\big|\big(B(x)\uu (x)\big)_i \big|\leqslant M_\aaa\qquad\hbox{for every $x\in\T^N$ and $i\in\ind$.}$
\end{itemize}
\end{prop}

\begin{dimo}
It suffices to prove the assertion for $\aaa=a\,\1$.
Let us set
\[
\beta_\star=\min_{1\leqslant i \leqslant m}\min_{x\in\T^N} b_{ii}(x),\qquad \beta^\star= \max_{1\leqslant i,j \leqslant m}\max_{x\in\T^N} |b_{ij}(x)|.
\]
Such  quantities are finite valued. Moreover, $\beta_\star$ is strictly positive in view of Proposition \ref{prop positive diagonal} and of the fact that $B(x)$ is, for every $x\in\T^N$, an irreducible coupling matrix with continuous coefficients.

Let us now fix $x\in\T^N$ and assume, without any loss of generality,
\begin{equation}\label{eq ordering}
u_1(x)\leqslant u_2(x)\leqslant\dots \leqslant u_m(x).
\end{equation}
First notice that, by subtracting   $\sum\limits_{j=1}^m b_{mj}(x)u_m(x)=0$ from both sides of equation \eqref{eq compactness} with $i=m$, one gets
\[
 \sum_{j\not=m}-b_{mj}(x)\big(u_m(x)-u_j(x)\big)\leqslant a,
\]
yielding
\[
\Big(u_m(x)-\max_{j\not=m} u_j(x)\Big) \sum_{j\not=m}-b_{mj}(x)\leqslant a.
\]
By exploiting \eqref{eq ordering} and the degenerate character of the matrix $B(x)$ we get
\begin{equation}\label{eq step}
0\leqslant u_m(x)-u_{m-1}(x)\leqslant\frac{a}{b_{mm}(x)}\leqslant\frac{a}{\beta_\star}\medskip.
\end{equation}
This proves assertion {\em (i)} when $m=2$. To prove it in the general case, we argue by induction: we assume the result true for $m$ and we prove  it  for $m+1$. To this aim, we restate equation \eqref{eq compactness} as
\begin{equation*}
 \sum_{j=1}^{m-1}b_{ij}(x)u_j(x)+\Big(b_{im}(x)+b_{i\, m+1}(x)\Big)u_m(x)+b_{i\, m+1}(x)\Big(u_{m+1}(x)-u_m(x)\Big)\leqslant a,
\end{equation*}
then we exploit \eqref{eq step} to get
\begin{equation}\label{eq system}
 \sum_{j=1}^{m-1}b_{ij}(x)u_j(x)+\Big(b_{im}(x)+b_{i\, m+1}(x)\Big)u_m(x)\leqslant a\left(1+\frac{\beta^\star}{\beta_\star}\right)
\end{equation}
for every $i\in\{1,\dots,m+1\}$. The irreducible character of $B(x)$ applied to the set $\I=\{m,m+1\}$ implies that 
\[
 b_{im}(x)+b_{i\, m+1}(x)>0
\]
for either $i=m$ or $i=m+1$, let us say $i=m$ for definitiveness. Assertion {\em (i)} now follows by applying the induction hypothesis to the system given by \eqref{eq system} with $i$ varying in $\ind$, the corresponding coupling matrix being still irreducible and degenerate.

To prove {\em (ii)} it suffices to note that, for every $i\in\ind$,
\begin{multline*}
-\big(B(x)\uu (x)\big)_i= -b_{ii}(x)u_i(x)+\sum_{j\not=i}\big(-b_{ij}(x)\big)u_j(x)\\
\quad\leqslant -b_{ii}(x)u_i(x) +\,\sum_{j\not=i} -b_{ij}(x)\big(u_i(x) + \|u_i-u_j\|_\infty \big) \\
\leqslant (m-1)\,\beta^\star\|u_i-u_j\|_\infty,
\end{multline*}
and the assertion follows from {\em (i)} and from hypothesis \eqref{eq compactness}.
\end{dimo}

\medskip
As a consequence, we derive the following result:

\begin{prop}\label{prop a priori Lip}
Let $\uu=(u_1,\dots,u_m)^T \in \left(\D{C}(\T^N)\right)^m$ be a  subsolution of
\eqref{wcoupled system} for some $\aaa\in\R^m$.
Then there exist  constants $C_\aaa$ and $\kappa_\aaa$, only depending on $\aaa$, on the Hamiltonians $H_1,\dots, H_{m}$ and on the coupling matrix $B(x)$, such that
\begin{itemize}
 \item[\em (i)]  \quad $\|u_i-u_j\|_\infty\leqslant C_\aaa\ \quad\qquad\hbox{for every $i,\,j\in\ind$;}$\medskip
 \item[\em (ii)] \quad $\uu$ is $\kappa_\aaa$--Lipschitz continuous in $\T^N$.
\end{itemize}
\end{prop}

\begin{dimo}
For each $i\in\ind$, the following inequalities hold in the viscosity sense:
\[
\mu + \big(B(x)\uu(x)\big)_i \leqslant H_i(x,Du_i)+\big(B(x)\uu(x)\big)_i \leqslant a_i\qquad\hbox{in $\T^N$},
\]
yielding
\[
\big(B(x)\uu(x)\big)_i\leqslant a_i-\mu\qquad\hbox{for every $x\in\T^N$}.
\]
In view of Proposition \ref{prop compactness} we get {\em (i)} and 
\[
\big|\big(B(x)\uu (x)\big)_i \big|\leqslant M_\aaa\qquad\hbox{for every $x\in\T^N$.}
\]
Plugging this inequality in \eqref{wcoupled system} we derive that $u_i$ is a viscosity subsolution of
\[
 H_i(x,Du_i)\leqslant a_i+M_\aaa\qquad\hbox{in $\T^N$}
\]
and assertion {\em (ii)} follows as well via a standard argument that exploits the coercivity of $H_i(x,p)$ in $p$, see for instance \cite{barles}.\medskip
\end{dimo}

Next, we establish a remarkable property of weakly coupled systems. 

\begin{prop}\label{prop pre-comparison}
Assume that $\vv,\,\uu\in\left(\D C(\T^N)\right)^m$ are, respectively a sub and a supersolution of the weakly coupled system \eqref{wcoupled system} for some $\aaa\in\R^m$. Let $x_0\in\T^N$ be such that  
\begin{equation*}
v_i(x_0)-u_i(x_0)=M:=\max_{1\leqslant i\leqslant m}\max_{\T^N}\, (v_i-u_i)\qquad\hbox{for some $i\in\ind$}.
\end{equation*}
Then $\vv(x_0)=\uu(x_0)+M\1$.
\end{prop}

\begin{dimo}
In view of Proposition \ref{prop a priori Lip}, we know that $\vv$ is Lipschitz continuous. 
Set 
\[
 \I=\big\{i\in\ind\,:\,\big(v_i(x_0)-u_i(x_0)\big)=M\,\big\}.
\]
We want to prove that $\I=\ind$. Indeed, if this were not the case, by the irreducible character of the matrix $B(x_0)$ there would exist $i\in\I$ and $k\not\in\I$ such that 
\begin{equation*}\label{eq1 pre-comparison}
 b_{ik}(x_0)<0.
\end{equation*}
We now make use of the method of doubling the variables  to reach a contradiction. For every $\eps>0$, we set 
\[
 \psi^\eps(x,y)=v_i(x)-u_i(y)-\frac{d(x,y)^2}{2\eps^2}-\frac{d(x,x_0)^2}{2},\qquad x,y\in\T^N.
\]
Let $M_\eps=\max\limits_{\T^N\times\T^N} \psi_\eps$ and denote by $(x_\eps,y_\eps)$ a point in $\T^N\times\T^N$ where such a maximum is achieved. By a standard argument in the theory of viscosity solution, see for instance Lemma 2.3 in \cite{barles}, the following properties hold:
\begin{equation}\label{estimates pre-comparison}
 x_\eps,y_\eps\to x_0,\quad \frac{d(x_\eps,y_\eps)}{\eps}\to 0 \qquad \hbox{as $\eps\to 0$.} 
\end{equation}
Furthermore,
\[
 p'_\eps:=\frac{x_\eps-y_\eps}{\eps^2}\in D^-u_i(y_\eps),\quad p_\eps:=p'_\eps-(x_\eps-x_0)\in D^+v_i(x_\eps)\qquad\hbox{for every $\eps>0$.}
\]
By the Lipschitz character of $v_i$ we derive that the vectors $\{\,p_\eps\,:\,\eps>0\,\}$ are equi--bounded, hence, up to subsequences and in view of the estimates \eqref{estimates pre-comparison}, we infer 
\[
 p_\eps,\,p'_\eps\to p_0\qquad\hbox{as $\eps\to 0$}
\]
for some vector $p_0\in\R^N$. We now use the fact that $\vv$ and $\uu$ are a sub and supersolution of \eqref{critical wcs c=0}, respectively, to get
\begin{eqnarray*}
 H_i(x_\eps,p_\eps)+\big(B(x_\eps)\vv(x_\eps)\big)_i&\leqslant& 0,\\
 H_i(y_\eps,p'_\eps)+\big(B(y_\eps)\uu(y_\eps)\big)_i&\geqslant& 0.
\end{eqnarray*}
By subtracting the above inequalities and by passing to the limit for $\eps\to 0$ we end up with 
\[
 \Big(B(x_0)\big(\vv(x_0)-\uu(x_0)\big)\Big)_i\leqslant 0,
\]
that is, since $i\in\I$ and the matrix $B(x_0)$ is degenerate, 
\[
 M\,b_{ii}(x_0)\leqslant \sum_{j\not=i} |b_{ij}(x_0)|\big(v_j(x_0)-u_j(x_0)\big)\leqslant M\,\sum_{j\not=i} |b_{ij}(x_0)|=M b_{ii}(x_0).
\]
Hence the above inequalities are equalities, in particular $v_k(x_0)-u_k(x_0)=M$ since $b_{ik}(x_0)\not=0$, in contrast with the fact that $k\not\in\I$.\medskip
\end{dimo}

\begin{definition}\rm
For every $\aaa\in\R^m$, we denote by $\hh (\aaa)$ the set of subsolutions of the weakly coupled system \eqref{wcoupled system}. We will more simply write $\hh(a)$ whenever $\aaa = a\1$ for some constant $a\in \R$.
\end{definition}

\begin{lemma}
The sets $\hh (\aaa)$ are convex and closed in $\left(\D C(\T^N)\right)^m$, and increasing with respect to the partial ordering on $\R^m$.
\end{lemma}
\begin{dimo}
Convexity and monotonicity are straightforward. The fact that the $\hh (\aaa)$ are closed is a direct consequence of stability of viscosity subsolutions.
\end{dimo}
\medskip

We now focus our attention to the case $\aaa=a\1$. As a direct consequence of the definition of the semigroup $\S(t)$, we get the following assertion:

\begin{prop}\label{prop fixed point}
Let $a\in\R$ and $\uu\in\big(\D{Lip}(\T^N)\big)^{m}$. Then $\uu$ is a viscosity solution of \eqref{wcoupled system} with $\aaa=a\,\1$ if and only if
\[
 \uu=\S(t)\uu+t\,a\1\quad\hbox{in $\T^N$}\qquad\hbox{for every $t>0$.}
\]
\end{prop}

We have the following characterization:

\begin{prop}\label{prop invariant}
Let $a\in\R$ and $\uu\in\big(\D{Lip}(\T^N)\big)^{m}$. The following facts are equivalent:
\begin{itemize}
 \item[\em (i)] \quad $\uu\in\hh(a)$;\smallskip
 \item[\em (ii)] \quad the map $t\mapsto \S(t)\uu+t\,a\1$ is non--decreasing on $[0,+\infty)$.
\end{itemize}
In particular, the
sets $\hh (a)$ are stable under the action of the semigroup $\S(t)$, in the sense that $\S(t) \big( \hh(a)\big) \subset \hh(a)$.
\end{prop}

The proof of this proposition is rather technical and it is postponed to the Appendix \ref{appendix semiconcave}.

\begin{definition}
The {\em critical value $c$} of the weakly coupled system \eqref{wcoupled system} is defined as 
\begin{equation}\label{def critical value}
c=\inf\{a\in\R\,:\,\hh(a)\not=\varnothing\,\}.
\end{equation}
\end{definition}

The following holds:

\begin{prop}\label{prop critical value}
The critical value $c$ is finite and $\hh(c)\not=\varnothing$.
\end{prop}

\begin{dimo}
By the growth assumptions on the Hamiltonians $H_i$ it is easily seen that the function $\mathbf u\equiv (0,\dots,0)^T$ is  a subsolution of \eqref{wcoupled system} for $a_0\1$ with $a_0\in\R$ big enough.

Let us proceed to show that $c$ is finite valued and that $\hh(c)\not=\varnothing$. Let $(a_n)_n$ be a decreasing sequence converging to $c$ and let ${\mathbf u_n}\in\hh(c_n)$ for each $n\in\N$. Up to neglecting the first terms, we can assume that $a_n\leqslant a_0$ for every $n\in\N$. Arguing as in the proof of Proposition \ref{prop a priori Lip}, we obtain that the following inequalities are satisfied in the viscosity sense:
\[
 \mu_i\leqslant H_i(x,Du^n_i)\leqslant a_n+M_{a_0}\qquad\hbox{in $\T^N$}
\]
for every $i\in\ind$ and $n\in\N$, showing that $c$ is finite.
We now exploit Proposition \ref{prop a priori Lip}: by the monotonicity of the sets $\hh(a)$ with respect to $a$, we infer that the functions ${\mathbf u_n}$ are equi--Lipschitz. Up to subtracting a vector of the form $k_n\1$ to each ${\mathbf u_n}$, we can furthermore assume that $u_1^n(0)=0$ for every $n\in\N$, yielding $\sup_n\|u_1^n\|_\infty\leqslant L$ for some $L\in\R$ by the equi--Lipschitz character of the sequence. Moreover, 
\[
 \|u_j^n-u_1^n\|_\infty\leqslant C_{a_0}\qquad\hbox{for every $j\in\ind$ and $n\in\N$,}
\]
yielding
\[
 \|u_j^n\|_\infty\leqslant C_{a_0}+L\qquad\hbox{for every $j\in\ind$ and $n\in\N$.}
\]
Up to subsequences, by the Arzela--Ascoli theorem, we infer that
\[
 \mathbf u^n \ucv \mathbf u\qquad\hbox{in $\T^N$}
\]
and ${\mathbf u}\in\hh(c)$ by stability of the notion of viscosity subsolution.
\end{dimo}\\

We now proceed to show that a weakly coupled system of the kind \eqref{wcoupled system} 
with $\aaa=a\1$ possesses solutions if and only if $a$ equals the critical value $c$. 

We start with a preliminary result.

\begin{prop}\label{prop strong comparison}
Let $B(x)$ be a continuous irreducible coupling matrix on $\T^N$ and let us assume that $B(x)$ is invertible for every $x\in\T^N$.
Let $\mathbf{v},\,\mathbf{u}\in\big(\D{C}(\T^N)\big)^m$ be, respectively, a sub and a supersolution of the weakly coupled system \eqref{wcoupled system},  for some $\aaa\in\R^m$.
Then
\[
 \vv(x)\leqslant \uu(x)\qquad\hbox{for every $x\in\T^N$.}
\]
\end{prop}

\begin{dimo}
Arguing as in the proof of Proposition \ref{prop a priori Lip}, we easily see that $\vv$ is Lipschitz. 
Set 
\[
M=\max_{1\leqslant i\leqslant m}\max_{\T^N}\, (v_i-u_i).
\]
We want to prove that $M\leqslant 0$. Assume by contradiction that $M>0$ and pick a point $x_0\in\T^N$ where such a maximum is attained. Set 
\[
 \I=\big\{i\in\ind\,:\,\big(v_i(x_0)-u_i(x_0)\big)=M\,\big\}.
\]
Arguing as in the proof of Proposition \ref{prop pre-comparison} we infer that
\begin{equation}\label{ineq strong comparison}
\Big(B(x_0)\big(\vv(x_0)-\uu(x_0)\big)\Big)_i\leqslant 0\qquad\hbox{for every $i\in\I$.}
\end{equation}
If $\I=\ind$, inequality \eqref{ineq strong comparison} is indeed an equality and this is in contradiction with the fact that $B(x_0)$ is invertible, in view of Proposition \ref{prop Ker}. 
If $\I\not=\ind$, we choose $i\in\I$ and $k\not\in\I$ such that $b_{ik}(x_0)<0$. From \eqref{ineq strong comparison} and the assumption that $M>0$ we infer that 
\[
 M\,b_{ii}(x_0)\leqslant \sum_{j\not=i} |b_{ij}(x_0)|\big(v_j(x_0)-u_j(x_0)\big)\leqslant M\,\sum_{j\not=i} |b_{ij}(x_0)|\leqslant M b_{ii}(x_0),
\]
which implies that $v_k(x_0)-u_k(x_0)=M$, in contrast with the fact that $k\not\in\I$.\medskip
\end{dimo}

The next result implies that solutions to a weakly coupled system of the kind \eqref{wcoupled system} 
with $\aaa=a\1$ may exist only if $a$ equals the critical value.

\begin{prop}\label{prop unique c}
Let $a,\,b\in\R$ and  $\vv,\,{\mathbf u} \in \left(\D{C}(\T^N)\right)^m$ such that the following inequalities are satisfied in the viscosity sense:
\begin{eqnarray*}
H_i(x,Dv_i)+\big(B(x)\vv(x)\big)_i&\leqslant& a\qquad\hbox{in $\T^N$}\\
H_i(x,Du_i)+\big(B(x)\uu(x)\big)_i&\geqslant& b\qquad\hbox{in $\T^N$}
\end{eqnarray*}
for every $i\in\ind$. Then $b\leqslant a$.
\end{prop}

\begin{dimo}
Let us assume by contradiction that $b>a$. Up to replacing $\vv$ with $\vv+k\1$ with $k>0$ big enough, we can assume
\[
 \vv >\uu \qquad\hbox{in $\T^N$.}
\]
Let $\eps>0$ such that \  $b-\eps>a+\eps$.\quad By continuity of the functions $\vv$ and $\uu$, we can find $\lambda>0$ such that
\[
 \|\lambda\,v_i\|_\infty,\,\|\lambda\,u_i\|_\infty<\eps\qquad\hbox{for every $i\in\ind$.}
\]
Then the following inequalities hold in the viscosity sense in $\T^N$:
\[
  H_i(x,Du_i)+\big((B(x)+\lambda\,\D{I})\uu(x)\big)_i
  >b-\eps
  >a+\eps
  >
  H_i(x,Dv_i)+\big((B(x)+\lambda\,\D{I})\vv(x)\big)_i.
\]
For ever $x\in\T^N$, the matrix $B(x)+\lambda\,\D I$ is irreducible, satisfies (C) and the sum of the elements of each of its rows is strictly positive, hence it is invertible in view of Proposition \ref{prop Ker}. By Proposition \ref{prop strong comparison} we conclude that 
\[
 \vv\leqslant\uu\qquad\hbox{in $\T^N$,}
\]
achieving a contradiction.
\end{dimo}\\

We are now able to prove existence of solutions for the critical system, following the lines of Fathi \cite{Fa}. 
This result has been already obtained in literature in similar settings by making use of the so called 
ergodic approximation, see \cite{mitaketran,leyetal}.

\begin{teorema}\label{weakKAM}
There exists a function $\uu\in\hh(c)$ that solves the weakly coupled system
\begin{equation}\label{eq critical wcs}
H_i(x,Du_i)+\big(B(x)\uu(x)\big)_i = c \quad\hbox{in $\T^N$}\qquad\hbox{for every $i\in\ind$}
\end{equation}
in the viscosity sense.
\end{teorema}

\begin{dimo}
We have already proved in Proposition \ref{prop critical value} that $\hh(c)\not=\varnothing$.
Let us introduce the quotient space $\hat\hh = \hh(c)\backslash \R \1$, where we identify critical subsolutions that differ by a constant vector belonging to $\R \1$.
Arguing as in the proof of Proposition \ref{prop critical value}, it is easily seen that $\hat \hh$ is compact for the  topology of uniform convergence. Indeed, it is isomorphic to the subset of $\hh(c)$ of subsolutions whose first component vanishes at the point $x=0$. Moreover, since the viscosity semigroup commutes with the addition of vectors of the form $\lambda \1$ and  leaves $\hh(c)$ stable, it induces a continuous semigroup, denoted $\hat S$, on $\hat\hh$.

By the Schauder--Tychonoff fixed point theorem (see \cite{dug}), $\hat S$ possesses a fixed point, that is, there exists an element $\hat \uu \in \hat\hh$ such that
$$\forall t \geqslant 0, \quad \hat S(t) \hat\uu =\hat \uu.$$
Lifting these relations to $\hh(c)$, we get
\[
\hbox{$\forall t \geqslant 0$\ there exists $c_t\in \R$ such that $\S(t)\uu=\uu+c_t\1,$}
\]
where $\uu$ is any element in the equivalence class of $\hat\uu$.
Since $\S$ is a semigroup, one readily realizes that the following relations are verified:
$$ c_{t+s}=c_t+c_s\qquad\hbox{for every $t,s>0$}.$$
Since $t\mapsto \S(t)\uu$ is continuous, we necessarily deduce that $c_t=-t\tilde c$ for all $t>0$ for some constant $\tilde c\in \R$.

The identity $\S(t)\uu=\uu-t\tilde c\1$, for all $t\geqslant 0$, implies that $\uu$ is a viscosity solution of \eqref{eq critical wcs} with $\tilde c$ in place of $c$, see Proposition \ref{prop fixed point}. But then $\tilde c=c$ in view of Proposition \ref{prop unique c} and the statement is proved.\medskip
\end{dimo}
\end{section}

\begin{section}{The Aubry set}\label{aubry mane}

In this section we start our qualitative analysis on the critical weakly coupled system, i.e. the system \eqref{wcoupled system} 
with $\aaa=c\1$, where $c$ is defined via \eqref{def critical value}. From now on we will always assume the  critical value $c$  to be equal to $0$. This renormalization is always possible  by replacing each $H_i$ with $H_i-c$. The critical weakly coupled system reads as 
\begin{equation}\label{critical wcs c=0}
H_i(x,Du_i)+\big(B(x)\uu(x)\big)_i=0\quad\hbox{in $\T^N$}\qquad\hbox{for every $i\in\ind$.}
\end{equation}
Solutions, subsolutions and supersolutions of  \eqref{critical wcs c=0} will be termed {\em critical} in the sequel. The family of critical subsolutions, we recall, is denoted by $\hh(0)$.

Our qualitative analysis on the critical weakly coupled system is based on the notion of  {\em Ma\~n\' e matrix}, defined in analogy with that of  the Ma\~n\' e potential.

\begin{definition}
For all $(x,y,i,j)\in \T^N\times \T^N\times\ind\times\ind$, we define
$$\Phi_{i,j}(y,x)=\sup_{\vv \in \hh(0)} v_i(x)-v_j(y).$$
\end{definition}

The following properties hold:

\begin{prop}\label{Man sous}
The Ma\~n\' e matrix verifies the following properties:
\begin{itemize}
\item[\em (i)] it is everywhere finite and Lipschtiz continuous;\medskip
\item[\em (ii)] $\Phi_{\cdot,j}(y,\cdot)\in \hh(0)$ \quad for every $(y,j)\in\T^N\times\ind$;\medskip
\item[\em (iii)] for every $(y,j)\in\T^N\times\ind$ and $\vv\in \hh(0)$,
\begin{equation*}
\vv-v_j(y)\1 \leqslant \Phi_{\cdot,j} (y,\cdot) \qquad\hbox{in $\T^N$},
\end{equation*}
namely $\Phi_{\cdot,j} (y,\cdot)$ is the maximal critical subsolution whose $j$--th component vanishes  at $y$;\medskip
\item[\em (iv)] the entries of the Ma\~n\'e matrix are linked by the following triangular inequality:
$$\Phi_{i,k}(x,z)\leqslant \Phi_{j,k}(x,y)+\Phi_{i,j}(y,z)$$
for every $i,j,k\in\ind$ and $x,y,z\in\T^N$.
\end{itemize}
\end{prop}

%

\begin{dimo}
The fact that the Ma\~ n\'  e matrix is well defined directly follows from Proposition \ref{prop a priori Lip}. Lipschitz continuity comes from the equi--Lipschitz character of critical subsolutions. 

The second assertion comes from the fact that $\Phi_{\cdot,j}(y,\cdot)$ is, for every fixed $(j,y)$, a supremum of critical subsolutions, hence itself a critical subsolution by Proposition \ref{prop Engler}.

The third point is a direct consequence of the definition.

The last point comes from the fact that $\Phi_{\cdot,j}(y,\cdot)$ is the greatest subsolution whose $j$--th component vanishes  at $y$. Since $\Phi_{\cdot,k}(x,\cdot)-\Phi_{j,k}(x,y)\1$ is a subsolution whose  $j$--th component vanishes  at $y$ we obtain that
$$\Phi_{\cdot,k}(x,\cdot)-\Phi_{j,k}(x,y)\1\leqslant \Phi_{\cdot,j}(y,\cdot),$$
which is the triangular inequality to be proved.\medskip
\end{dimo}

%
%

As in the case of a single critical equation, the Ma\~ n\' e vectors are ``almost'' critical solutions, in the sense precised below:
\begin{prop}\label{prop Mane matrix}
Let $y_0\in\T^N$ and $i_0\in\ind$. Then the function $\uu =\Phi_{\cdot,i_0}(y_0,\cdot)$ satisfies  
\[
 H_{i_0}(x,Du_{i_0})+\big(B(x)\uu(x)\big)_{i_0}=0\quad\hbox{in $\T^N\setminus\{y_0\}$.}
\]
and 
\[
H_i(x,Du_i)+\big(B(x)\uu(x)\big)_i=0\quad\hbox{in $\T^N$\qquad for every $i\not=i_0$}
\]
in the viscosity sense. 
\end{prop}

\begin{dimo}
We argue by contradiction, following the classical argument of \cite{Fa} for the classical Ma\~n\'e potential.

Let $(i,y)$ be such that either $i\neq i_0$ or $y\neq y_0$. Let us assume that the viscosity supersolution condition is violated at $(i,y)$. This means that there exists
 a $ C^1$ function $\psi$ such that  $\psi(x)\leqslant \Phi_{i,i_0}(y_0,x)$ for all $x$, with equality if and only if $x=y$, and 
\[
H_i\big(x,D\psi(y)\big)+\big(B(y)\Phi_{\cdot,i_0}(y_0,y)\big)_i <0.
\]
Since $\psi$ is $ C^1$, and $B(\cdot)$ and $\Phi_{\cdot,i_0}(x_0,\cdot)$ are continuous, it is clear that this strict inequality continues to hold in a neighborhood of $y$. We infer that it is possible to find $\eps>0$ small enough such that the function $w_i :=\max\{\Phi_{i,i_0}(y_0,\cdot),\psi+\eps\}$ verifies
$$
H_i\big(x,Dw_i(x)\big)+\big(B(x)\ww(x)\big)_i \leqslant 0\qquad\hbox{for a.e. $x\in\T^N$,}
$$
where $\ww$ is the vector whose $i$--th coordinate is $w_i$ and whose other coordinates are those of $\Phi_{\cdot,i_0}(y_0,\cdot)$. In the case when $i=i_0$ and $y\neq y_0$, we choose $\eps>0$ small enough in such a way that $w_i(y_0)=\Phi_{i,i_0}(y_0,y_0)=0$. Moreover, for every $j\not=i$, 
\[
H_j\big(x,Dw_j(x)\big)+\big(B(x)\ww(x)\big)_j \leqslant 0\qquad\hbox{for a.e. $x\in\T^N$,}
\]
as it is easily seen from the fact that $b_{ji}(\cdot)\leqslant 0$ in $\T^N$ and $w_i\geqslant \Phi_{i,i_0}(y_0,\cdot)$. 

We have thus shown that $\ww$ is a critical subsolution with $w_{i_0}(y_0)=0$, $\ww\geqslant \Phi_{i,i_0}(y_0,\cdot)$ and $\ww\not\equiv \Phi_{i,i_0}(y_0,\cdot)$, thus contradicting the maximality of $\Phi_{i,i_0}(y_0,\cdot)$ amongst subsolutions whose $i_0$--th coordinate vanishes at $y_0$.\medskip
\end{dimo}

Next, we show a strong invariance property enjoyed by the rows of the Ma\~n\'e matrix. 

\begin{prop}
Let $i,j\in\ind$ and $y\in\T^N$. If $\Phi_{\cdot,i}(y,\cdot)$ is a critical solution on $\T^N$, then $\Phi_{\cdot,j}(y,\cdot)$ is too. 
\end{prop}

\begin{dimo}
Let us set $\vv:=\Phi_{\cdot,j}(y,\cdot)$ and $\uu:=\Phi_{\cdot,i}(y,\cdot)+\Phi_{i,j}(y,y)\,\1$. In view of Proposition \ref{prop Mane matrix}, we only need to show that 
\[
  H_{j}(y,p)+\big(B(y)\vv(y)\big)_{j}\geqslant 0\quad\hbox{for every $p\in D^- v_j(y)$.}
\]
According to Proposition \ref{Man sous}, $\vv\leqslant \uu$ in $\T^N$ and $v_i(y)=u_i(y)$.
The functions $\vv$ and $\uu$ being respectively a critical subsolution and a solution,  we can apply Proposition \ref{prop pre-comparison} to infer that $\vv(y)=\uu(y)$. This also implies that $D^- v_j(y)\subseteq D^- u_j(y)$. Exploiting again the fact that $\uu$ is a critical solution we finally get
\[
 0\leqslant H_{j}(y,p)+\big(B(y)\uu(y)\big)_{j}=H_{j}(y,p)+\big(B(y)\vv(y)\big)_{j}\quad\hbox{for every $p\in D^- v_j(y)$.}
\]
\end{dimo}

In view of the previous proposition, the following definition is well posed:

\begin{definition}
The Aubry set $\A$ for the weakly coupled system \eqref{critical wcs c=0} is the  set defined as
\[
 \A=\left\{ y\in\T^N\,:\,\Phi_{\cdot,i}(y,\cdot)\ \hbox{is a critical solution}\,\right\},
\]
where $i$ is any fixed index in $\ind$. 
\end{definition}

By the continuity of the Ma\~n\'e matrix and the stability of the notion of viscosity solution, it is easily seen that $\A$ is closed. The analysis we are about to present will show that the Aubry set is nonempty: as in the corresponding critical scalar case, we will see that $\A$ is the set where the obstruction to the existence of globally strict critical subsolutions concentrates.

\begin{definition}
Let $\vv\in\hh(0)$. We will say that $v_i$ is {\em strict at $y\in\T^N$} if there exist an open neighborhood $V$ of $y$ and $\delta>0$ such that   
\begin{equation*}\label{def i,y strict}
H_i\big(x,Dv_i(x)\big)+\big(B(x)\vv(x)\big)_i <-\delta \quad\hbox{for a.e. $x\in V$}. 
\end{equation*}

We will say that $v_i$ is strict in an open subset $U$ of $\T^N$ if it is strict at $y$ for every $y\in U$.
\end{definition}
%

We start by establishing an auxiliary result that will be needed in the sequel.

\begin{lemma}\label{lemma inferno}
Let $\ww\in\hh(0)$ such that $w_i$ is strict at $y\in\T^N$. Then there exists $\widetilde\ww\in\hh(0)$ such that $\widetilde w_i$ is ${ C}^\infty$ and strict in a neighborhood of $y$. 
\end{lemma}

\begin{dimo}
By hypothesis, there exist  $r>0$ and $\delta>0$ such that
\begin{equation*}
H_i\big(x,Dw_i(x)\big)+\big(B(x)\ww(x)\big)_i <-\delta \quad\hbox{for a.e. $x\in B_{2r}(y)$.}
\end{equation*}
Let $\phi:\T^N\to [0,1]$ be a ${C}^\infty$--function, compactly supported in $B_r(y)$ and such that $\phi\equiv 1$ in $B_{r/2}(y)$. 
Let us denote by $\kappa$ a Lispchitz constant for the critical subsolutions and by $\omega$ a continuity modulus of $H_i$ in $\T^N\times B_R$ for some fixed $R>\kappa+\|D\phi\|_\infty$. Let $(\rho_n)_n$ be a sequence of standard mollifiers  on $\R^N$ and define
\[
 \psi_n(x)=(\rho_n * w_i)(x)+\|\rho_n * w_i-w_i\|_\infty,\qquad x\in\T^N.
\]
Note that $\psi_n\geqslant w_i$ in $\T^N$ for every $n\in\N$ and 
\[
 d_n:=\|\psi_n-w_i\|_\infty\to 0\qquad\hbox{as $n\to +\infty$.}
\]
Up to neglecting the first terms, we furthermore assume that all the $d_n$ are less than 1. 
For every $n\in\N$, we define a function $\widetilde\ww^n\in \big(\D{Lip}(\T^N)\big)^m$ by setting
\[
 \widetilde w^n_j(x)=w_j(x)\quad\hbox{if $j\not=i$,}\qquad \widetilde w^n_i(x)=\phi(x)\psi_n(x)+\big(1-\phi(x)\big)w_i(x)
\]
for every $x\in\T^N$. It is apparent by the definition that $\widetilde w_i^n$ is of class $ C^\infty$ in $B_{r/2}(y)$.
Moreover the functions $(\widetilde w^n_i)_n$, and hence the $(\widetilde \ww^n)_n$, are equi--Lipschitz. Indeed, for almost every $x\in\T^N$,  
\begin{equation}\label{eq inferno equi}
  D\widetilde w^n_i(x)=\phi(x) D\psi_n(x) +\big(1-\phi(x)\big)Dw_i(x)+\big(\psi_n(x)-w_i(x)\big)D\phi(x)
\end{equation}
that is \quad $\|D\widetilde w^n_i\|_\infty\leqslant \kappa+\|D\phi\|_\infty$. 
We want to show that $n$ can be chosen sufficiently large in such a way that $\widetilde\ww^n\in\hh(0)$ and 
\begin{equation}\label{eq2 inferno}
H_i\big(x,D\widetilde w^n_i(x)\big)+\big(B(x)\widetilde\ww^n(x)\big)_i <-\frac{2}{3}\,\delta \qquad\hbox{for a.e. $x\in B_{r}(y)$.}
\end{equation}
We first note that, since $\widetilde w^n_i\geqslant w_i$ and $b_{ji}\leqslant 0$ in $\T^N$ for every $j\not=i$, we have 
\begin{equation}\label{eq4 inferno}
H_j(x,D\widetilde w^n_j(x))+\big(B(x)\widetilde\ww^n(x)\big)_j\leqslant 0\quad\hbox{in $\T^N$}\qquad\hbox{for every $j\not=i$.}
\end{equation}
Moreover, since $\widetilde \ww^n$ agrees with $\ww$ outside $B_r(y)$, in order to show that $\widetilde\ww^n$ satisfies \eqref{eq4 inferno} 
also for $j=i$, it will be enough,  by the convexity of $H_i$, to prove \eqref{eq2 inferno}. 

To this aim, we start by noticing that 
\begin{equation}\label{eq1 inferno}
 H_i\big(x,D\widetilde w^n_i(x)\big) \leqslant \phi(x) H_i\big(x,D\psi_n(x)\big)+\big(1-\phi(x)\big)H_i\big(x,Dw_i(x)\big)+\omega\left(d_n\,\|D\phi\|_\infty\right)
\end{equation}
for almost every $x\in\T^N$, in view of \eqref{eq inferno equi} and of the convexity of $H_i$.  
By Jensen's inequality, for every $n>1/r$ and every $x\in B_{r}$ we have
\begin{eqnarray}\label{eq3 inferno}
H_i\big(x,D\psi_n(x)\big)
  &=&
  H_i\left(x,\int_{B_{1/n}} Dw_i(x-y)\rho_n(y)\,\dd y\right)\nonumber\\
  &\leqslant&
  \int_{B_{1/n}} H_i\big(x,Dw_i(x-y)\big)\rho_n(y)\,\dd y\nonumber\\
  &\leqslant&
  \omega(1/n)+\int_{B_{1/n}} H_i\big(x-y,Dw_i(x-y)\big)\rho_n(y)\,\dd y\nonumber\\
  &\leqslant&
  -\int_{B_{1/n}} \big(B(x-y)\ww(x-y)\big)_i\,\rho_n(y)\,\dd y-\delta+\omega(1/n)\nonumber\\
  &\leqslant&
  -\big(B(x)\widetilde\ww^n(x)\big)_i-\delta+\omega(1/n)+\eps_n,
\end{eqnarray}
where 
\[
 \eps_n:=\sup_{|z|\leqslant 1/n}\big\|\big(B(\cdot+z)\ww(\cdot+z)-B(\cdot)\widetilde\ww^n(\cdot)\big)_i\big\|_\infty.
\]
Since $\widetilde\ww^n\ucv\ww$ in $\T^N$ and all these functions are equi--Lipschitz, it is easily seen that $\lim_n\eps_n=0$. 
Furthermore
\begin{equation}\label{eq5 inferno}
  H_i\big(x,Dw_i(x)\big)
  \leqslant
  -\big(B(x)\widetilde\ww^n(x)\big)_i
  -\delta+\eps_n
  \quad
  \hbox{for a.e. $x\in B_{r}(y)$.}
  \end{equation}
We now choose $n>1/r$ sufficiently large such that
\[
\omega\left(d_n\,\|D\phi\|_\infty\right)+\omega(1/n)+\eps_n<\frac{\delta}{6}
\]
and plug \eqref{eq3 inferno} and \eqref{eq5 inferno} into \eqref{eq1 inferno} to finally get \eqref{eq2 inferno}. 
The assertion follows by setting $\widetilde\ww:=\widetilde\ww^n$ for such an index $n$.\medskip
\end{dimo}

The next proposition shows that the $i$--th component of any critical subsolution fulfills the supersolution test on $\A$. 

\begin{prop}\label{prop supersol on A_i}
Let  $y\in\A$. Then, for every $i\in\ind$ and  $\ww\in\hh(0)$, 
\begin{equation}\label{claim supersol on A_i}
H_i(y,p)+\big(B(y)\ww(y)\big)_i=0\qquad\hbox{for every $p\in D^-w_i(y)$.} 
\end{equation}
\end{prop}


\begin{dimo}
Pick $\ww\in\hh(0)$ and set $\uu=\Phi_{\cdot,i}(y,\cdot)+w_i(y)\1$. According to Proposition \ref{Man sous}, $\ww\leqslant\uu$ and, by definition of $\uu$, $w_i(y)=u_i(y)$, in particular $D^-w_i(y)\subseteq D^-u_i(y)$.  Now we exploit the fact that $\uu$ and $\ww$ are a critical solution and  subsolution, respectively:  from Proposition \ref{prop pre-comparison} we infer that $\ww(y)=\uu(y)$, while Proposition \ref{prop subsol equivalent def} implies 
\[
 0\geqslant H_i(y,p)+\big(B(y)\ww(y)\big)_i= H_i(y,p)+\big(B(y)\uu(y)\big)_i\geqslant 0\quad\hbox{$\forall p\in D^-w_i(y)$.}
\]
Hence all the inequalities must be equalities and the statement follows.\medskip 
\end{dimo}

A converse of this result is given by the following 

\begin{prop}\label{prop i,y strict}
Let $i\in\ind$. The following facts are equivalent:
\begin{itemize}
 \item[{\em (i)}]\ $y\not\in\A$;\medskip
 \item[{\em (ii)}]\ there exists $\ww\in\hh(0)$ such that $w_i$ is strict  at $y$.
\end{itemize}
Moreover, $w_i$ can be taken of class $C^1$ in a neighborhood of $y$.
\end{prop}

\begin{dimo}
Let us assume {\em (i)}. Since $y\not\in\A$, the supersolution test for $\Phi_{\cdot,i}(y,\cdot)$ is violated at $(i,y)$. This means that there exists a $ C^1$ function $\psi$ such that  $\psi(x)\leqslant \Phi_{i,i}(y,x)$ for all $x$, with equality if and only if $x=y$, and 
\[
H_i\big(x,D\psi(y)\big)+\big(B(y)\Phi_{\cdot,i}(y,y)\big)_i <0.
\]
We define a function $\ww\in\big(\D{Lip}(\T^N)\big)^m$ by setting 
\[
w_i(\cdot)=\max\{\Phi_{i,i}(y,\cdot),\psi+\eps\},\qquad w_j(\cdot)=\Phi_{j,i}(y,\cdot)\quad\hbox{for $j\not=i$.}
\]
Arguing as in the proof of Proposition \ref{prop Mane matrix} we see that it is possible to choose $\eps>0$ in a such a way that $\ww$ is a critical subsolution. Moreover, since $w_i$ agrees with $\psi+\eps$ in a neighborhood of $y$, there exist $\delta>0$ and an open neighborhood $W$ of $y$ such that $w_i$ is of class $ C^1$ in $W$ and 
\[
 H_i\big(x,Dw_i(x)\big)+\big(B(x)\ww(x)\big)_i < -\delta \quad\hbox{for every $x\in W$}.
\]

Conversely, let assume {\em (ii)}. According to Lemma \ref{lemma inferno}, there exists $\widetilde\ww\in\hh(0)$ such that $\widetilde w_i$ is smooth and strict in a neighborhood of $y$, in particular   
\[
  H_i\big(y,D\widetilde w_i(y)\big)+\big(B(y)\widetilde\ww(y)\big)_i< 0. 
\]
In view of Proposition \ref{prop supersol on A_i} we conclude that $y\not\in\A$.
\end{dimo}

\begin{oss}
Proposition \ref{prop supersol on A_i} expresses the fact, roughly speaking, that the $i$--th component of a critical subsolution cannot be strict at $y$. However, since the supersolution test \eqref{claim supersol on A_i} is void when $D^-u_i(y)$ is empty, this fact cannot be directly used to prove the equivalence stated in Proposition \ref{prop i,y strict}. This is the reason why we needed the regularization Lemma \ref{lemma inferno}.
\end{oss}

We proceed by proving a global version of the previous proposition. We give a definition first.

\begin{definition}
Let $\vv\in\hh(0)$.
We will say that $\vv$ is {\em strict at $y$} if $v_i$ is strict at $y$ for every $i\in\ind$. 
We will say that $\vv$ is strict in an open subset $U$ of $\T^N$ if it is strict at $y$ for every $y\in U$.  
\end{definition}

\begin{teorema}\label{teo Aubry}
There exists $\vv\in\hh(0)$ which is strict in $\T^N\setminus\A$. In particular, 
the Aubry set $\A$ is closed and nonempty. 
\end{teorema}

\begin{dimo}
Fix $i\in\ind$. We first construct a critical subsolution $\vv^i$ whose $i$--th component is strict in $\T^N\setminus\A$. According to Proposition \ref{prop i,y strict}, for every $y\in\T^N\setminus\A$ there exist an open neighborhood $W_y$ of $y$, a critical subsolution  $\ww^y$ and $\delta_y>0$ such that 
\begin{equation}\label{claim local strict}
H_i\big(x,Dw^y_i(x)\big)+\big(B(x)\ww^y(x)\big)_i < -\delta_y \quad\hbox{for a.e. $x\in W_y$}
\end{equation}
The family $\{W_y\,:\,y\in\T^N\setminus\A\,\}$  is an open covering of $\T^N\setminus\A$, from which we can extract a countable covering $(W_n)_n$ of $\T^N\setminus\A$. For each $n\in\N$, let us denote by $(\ww^n,\delta_n)$ the corresponding pair in $\hh(0)\times (0,+\infty)$ that satisfies \eqref{claim local strict} in $W_n$.  
%
Up to subtracting to each critical subsolution $\ww^n$ a vector of the form $k_n\1$, we can moreover assume that $w^n_1(0)=0$. Hence the functions $\ww^n$ are componentwise equi--Lipschitz and equi--bounded in view of Proposition \ref{prop a priori Lip}, in particular
the function 
\[
 \vv^i(x)=\sum_{n=1}^\infty \frac{1}{2^n}\,\ww^n(x),\qquad\hbox{$x\in\T^N$}
\]
is well defined and belongs to $\big(\D{Lip}(\T^N)\big)^m$. By convexity of the Hamiltonians, for almost every $x\in\T^N$ we get 
\begin{eqnarray*}
  H_i\big(x,Dv^i_i(x)\big)+\big(B(x)\vv^i(x)\big)_i
  \leqslant
  \sum_{n=1}^\infty \frac{1}{2^n}\Big( H_i\big(x,Dw^n_i(x)\big)+\big(B(x)\ww^n(x)\big)_i\Big)
  \leqslant
 0.  
\end{eqnarray*}
Moreover, the above inequalities hold with $-{\delta_k}/{2^k}$ in place of $0$ almost everywhere in $W_k$, for every $k\in\N$. This shows that 
$\vv^i$ is a critical subsolution, strict in $\T^N\setminus\A$.
Now set 
\[
 \vv(x)=\sum_{i=1}^m \frac{1}{m} \vv^i(x),\qquad x\in\T^N.
\]
A similar argument shows that $\vv$ is a critical subsolution that satisfies the assertion. 

If $\A=\varnothing$, by compactness we would have $\hh(-\delta)\not=\varnothing$ for some $\delta>0$, contradicting the definition of the critical value $c=0$. 
\medskip
\end{dimo}


In view of Proposition \ref{prop i,y strict}, we have the following characterization:

\begin{teorema}\label{teo char A}
Let $y\in\T^N$. The following are equivalent facts:
\begin{itemize}
\item[\em (i)] $y\notin\A$;\medskip
\item[\em (ii)] there exists $\ww\in\hh(0)$ which is strict at $y$;\medskip
\item[{\em (iii)}] there exists $\ww\in\hh(0)$ and $i\in\ind$ such that $w_i$ is strict at $y$.
\medskip
\end{itemize}
\end{teorema}

We end this section by extending to weakly coupled systems a result which is well known in the case of a single critical equation. 

\begin{prop}\label{prop S(t) su A}
\quad$\displaystyle{\A=\bigcap_{\ww\in\hh(0)} \left\{ y\in\T^N\,:\,\big(\S(t)\ww\big)(y)=\ww(y)\ \quad\hbox{for every $t>0$}\,\right\}.}$
\end{prop}

\begin{dimo}
Let us denote by $\A'$ the set appearing at the right--hand side of the above equality. Fix a point $y\in\A$ and let $\ww$ be any critical subsolution. For every fixed index $i\in\ind$, the function $\uu^i=\Phi_{\cdot,i}(y,\cdot)+w_i(y)\1$ satisfies $\ww\leqslant \uu^i$ in $\T^N$ and $w_i(y)=u_i(y)$. Moreover, $\uu^i$ is a critical solution, hence it is a fixed point for the semigroup $\S(t)$ by Proposition \ref{prop fixed point}. By monotonicity of the semigroup, we have 
\[
 w_i(y)\leqslant \big(\S(t)\ww\big)_i(y) \leqslant \big(\S(t)\uu^i\big)_i(y)=u^i_i(y)\quad\hbox{for every $t>0$,}
\]
hence all the inequalities must be equalities, in particular $\big(\S(t)\ww\big)_i(y)=w_i(y)$ for every $t>0$. This being true for every $i\in\ind$ and $\ww\in\hh(0)$, we conclude that $y\in\A'$.

To prove the converse inclusion, pick $y\in\A'$ and assume by contradiction that $y\not\in\A$. Fix $i\in\ind$ and take a 
critical subsolution $\vv$ such that $v_i$ is of class $C^1$ and strict in a neighborhood of $y$, according to Proposition \ref{prop i,y strict}. By Proposition \ref{prop invariant}, the map $(t,x)\mapsto v_i(x)$ is a subtangent to $\big(\S(t)v)_i(x)$ at $(t_0,y)$ for every $t_0>0$ and since the latter is a solution of the evolutionary system \eqref{evo wcoupled system} we get 
\[
 H_i(y,Dv_i(y))+\big(B(y)\S(t_0)\vv(y)\big)_i\geqslant 0.
\]
By sending $t_0\to 0^+$ we get a contradiction with the fact that $v_i$ is strict at $y$. 
\bigskip
\end{dimo}

\end{section}
\begin{section}{Regularization}\label{regularization}
The aim of this section is to show how a strict critical subsolution can be regularized outside the Aubry set. In the case of a single critical equation, it is known that such procedure can be performed in such a way that the output is a strict critical subsolution which is, in addition, of class $ C^1$ on the whole torus, see \cite{BernardC11,FSC1,FS05}. This result holds for Hamiltonians that are locally Lipschitz in $(x,p)$ and strictly convex in $p$ and the proof relies on the following two facts: first, any critical subsolution is differentiable on the Aubry set and, second, its gradient is independent of the specific subsolution chosen. This latter rigidity property holds for weakly couples systems too, as we will show at the end of the current section. What prevents us to extend to systems the existence of $ C^1$ strict critical subsolutions  is the lack of information on differentiability properties of critical subsolutions on the Aubry set.\smallskip    

We first deal with the regularization issue. The tools are not new and are mainly borrowed from \cite{FSC1,FS05}. However, we provide a proof for the reader's convenience.

We start with a local regularization argument.

\begin{lemma}\label{lemma local regularize}
Let $\uu\in\hh(0)$ and assume that, for some $r>0$, $\delta>0$ and $y\in\T^N\setminus\A$ and for every $i\in\ind$,
\[
H_i\big(x,Du_i(x)\big)+\big(B(x)\uu(x)\big)_i <-\delta \qquad\hbox{for a.e. $x\in B_{2r}(y)$.} 
\]
Then, for every $\eps>0$, there exists $\uu^\eps\in\hh(0)$ such that 
\begin{itemize}
 \item[\em (i)] \quad $\|\uu^\eps-\uu\|_\infty<\eps$;\medskip
 \item[\em (ii)] \quad $\uu^\eps=\uu$\quad in $\T^N\setminus B_r(y)$;\medskip
 \item[\em (iii)] \quad $\uu^\eps$ is of class $ C^\infty$ in $B_{r/2}(y)$ and satisfies
\begin{equation}\label{claim local regularization}
 H_i\big(x,Du^\eps_i(x)\big)+\big(B(x)\uu^\eps(x)\big)_i <-\frac{2}{3}\,\delta \qquad\hbox{for every $x\in B_{r/2}(y)$.}
\end{equation}
\end{itemize}
\end{lemma}

\begin{dimo}
Let $\phi:\T^N\to [0,1]$ be a ${C}^\infty$ function, compactly supported in $B_r(y)$ and such that $\phi\equiv 1$ in $B_{r/2}(y)$. 
Let $(\rho_n)_n$ be a sequence of standard mollifiers  on $\R^N$.
For every $n\in\N$, we define a function $\ww^n\in \big(\D{Lip}(\T^N)\big)^m$ by setting
\[
  w^n_i(x)=\phi(x)(\rho_n * u_i)(x)+\big(1-\phi(x)\big)u_i(x)\qquad\hbox{for every $x\in\T^N$ and $i\in\ind$.}
\]
It is apparent by the definition that $\ww^n$ is of class $  C^\infty$ in $B_{r/2}(y)$ and agrees with $\uu$ outside $B_r(y)$. Arguing as in the proof of Lemma \ref{lemma inferno}, we see that it is possible to choose $n$ large enough in such a way that 
$\ww^n$ is a critical subsolution and satisfies \eqref{claim local regularization}. Since $\ww^n\ucv\uu$ in $\T^N$, the assertion follows by setting $\uu^\eps:=\ww^n$ for a sufficiently large $n$.\medskip
\end{dimo}

We now prove the announced regularization result. 

\begin{teorema}\label{teo smooth subsol}
There exists a critical subsolution which is strict and $  C^\infty$ in $\T^n\setminus \A$. 
More precisely, for every critical subsolution $\vv$ which is strict in $\T^N\setminus\A$ and for every $\eps>0$, there exists $\vv^\eps\in\hh(0)$ such that
\begin{itemize}
 \item[\em (i)] \quad $\|\vv^\eps-\vv\|_\infty<\eps$;\medskip
 \item[\em (ii)] \quad $\vv^\eps=\vv$\quad on $\A$;\medskip
 \item[\em (iii)] \quad $\vv^\eps$ is $ C^\infty$ and strict in $\T^n\setminus \A$.\medskip
\end{itemize}
Moreover, the set of such smooth and strict subsolutions is dense in $\hh(0)$. 
\end{teorema}

\begin{dimo}
We first show how to regularize a subsolution which is strict outside the Aubry set. Let $\vv$ be such a subsolution (given by Theorem \ref{teo Aubry}) and fix $\eps>0$. Since $\vv$ is strict in $\T^N\setminus\A$, there exists a continuous and non--negative function $\delta:\T^N \to \R$ with $\delta^{-1}\left(\{0\}\right) = \A$ such that 
$$
H_i(x,Dv_i)+\big(B(x)\vv(x)\big)_i \leqslant -\delta(x) \quad\hbox{in $\T^N$}
$$
for every $i\in\ind$. Clearly, it is not restrictive to assume that the inequality $\delta(x)\leqslant \min\{\eps/2 , d(x,\A)^2\}$ holds  for every $x\in\T^N$. In view of Lemma \ref{lemma local regularize}, we can find a locally finite covering $(U_n)_n$ of $\T^N\setminus\A$ by open sets compactly contained in $\T^N\setminus\A$ and a sequence $(\uu^n)_n$ of critical subsolutions such that each $\uu^n$ is $  C^\infty$ in $U_n$ and satisfies 
\begin{eqnarray}
H_i(x,Du^n_i)+\big(B(x)\uu^n(x)\big)_i \leqslant -\frac{2}{3}\,\delta(x) \qquad\hbox{for every $x\in U_n$,}\label{eq u_n strict}\nonumber\\
|\uu^n(x)-\vv(x)|\leqslant\delta(x)\qquad\hbox{for every $x\in\T^N$.}\label{eq u_n control}
\end{eqnarray}
Set 
\[
 \delta_n:=\inf_{x\in U_n} \delta(x)\qquad\hbox{for every $n\in\N$}
\]
and choose a sequence $(\eta_n)_n$ in $(0,1)$ such that, for every $x\in\T^N$ and $n\in\N$, the following holds: 
\begin{equation}\label{def eta_n}
|H(x,p)-H(x,p')|<\frac{\delta_n}{6}\qquad\hbox{for all $p,p'\in B_{\kappa+1}$ with $|p-p'|<\eta_n$,}
\end{equation}
where $\kappa$ denotes a common Lipschitz constant for the critical subsolutions, in particular for all the $\uu^n$.  
Last, take a smooth partition of unity $(\varphi_n)_n$ subordinate to $(U_n)_n$ and choose the functions $\uu^n$ in such a way that the quantities $\|\uu^n-\vv\|_\infty$, which can be be made as small as desired, satisfy
\begin{equation}\label{eq pain}
 \sum_{\substack{k\in \N \\ U_{k}\cap U_n\neq \varnothing }} \|\uu^k-\vv\|_\infty \left\| D\varphi_{k}\right\|_\infty<\eta_n\qquad\hbox{for every $n\in\N$}.
\end{equation}
That is always possible since the covering $(U_n)_n$ is locally finite. 

We now define $\vv^\eps:\T^N\to\R^m$ by setting
\[
 \vv^\eps(x)=\sum_{n=1}^\infty \varphi_n(x)\uu^n(x)\quad\hbox{in $\T^N\setminus\A$}
\qquad
\hbox{and}
\qquad
\vv^\eps(x)=\vv(x)\quad\hbox{on $\A$.} 
\]
By definition, $\vv^\eps$ satisfies assertion {\em (ii)} and is $ C^\infty$ in $\T^N\setminus\A$. From \eqref{eq u_n control} we infer that $|\vv^\eps(x)-\vv(x)|\leqslant\delta(x)$ in $\T^N\setminus\A$, which shows at once that $\vv^\eps$ is continuous in $\T^N$ and that it satisfies assertion {\em (i)}. Moreover, by taking into account \eqref{eq pain} and the fact that $\sum D\varphi_k\equiv 0$, one obtains, for every  $x\in U_n$ and $ i\in \ind$, that 
\begin{equation}\label{eq stima}
\Big| Dv^\eps_i(x)-\sum_{\substack{k\in \N \\ U_{k}\cap U_n\neq \varnothing }}\varphi_{k}(x)Du_i^{k}(x) \Big| 
=
\Big| \sum_{\substack{k\in \N \\ U_{k}\cap U_n\neq \varnothing }}\big(u_i^{k}(x)-v(x)\big) D\varphi_{k}(x) \Big|
<
\eta_n,
\end{equation}
in particular 
\[
 |Dv^\eps_i(x)|
  \leqslant 
  \eta_n+
  \sum_{\substack{k\in \N \\ U_{k}\cap U_n\neq \varnothing }}\varphi_{k}(x)|Du_i^{k}(x)|
  \leqslant
  1+\kappa.
\]
We infer that  $\vv^\eps$ is Lipschitz--continuous in $\T^N$. In order to prove that $\vv^\eps$ is a critical subsolution and is strict in $\T^N\setminus\A$, it will be enough to show that  
$$
H_i(x,Dv^\eps_i(x))+\big(B(x)\vv^\eps(x)\big)_i \leqslant -\frac{\delta(x)}{2} \qquad\hbox{for a.e. $x\in\T^N$},
$$
for all $i\in \ind$.

Recall the Lipschitz functions $\vv^{\eps}$ and $\vv$ coincide on the Aubry set. Setting $\ww^{\eps} = \vv^{\eps} - \vv$, we infer that if $x_0\in \A$, then $|\ww^{\eps}(x)-\ww^{\eps}(x_0)| \leqslant d(x,\A)^2 \leqslant d(x,x_0)^2$. Hence $\ww^{\eps}$ is differentiable on $\A$ with vanishing differential and $D\vv^\eps(x)=D\vv(x)$  for almost every $x\in\A$. Hence, it suffices to establish the claim in the complementary of $\A$. To this aim, 
by recalling the definition of $\eta_n$ and by making use of \eqref{eq stima} and of Jensen inequality, we get that, for every $x\in U_n$ and $i\in \ind$,
\begin{align*}
H_i\big(x,Dv^\eps_i(x)\big)&+\big(B(x)\vv^\eps(x)\big)_i 
\leqslant H_i \Big( x, \sum_{\substack{k\in \N \\ U_{k}\cap U_n\neq \varnothing }}\varphi_{k}(x)Du_i^{k}(x) \Big)+\frac{\delta_n}{6}\\
&\qquad\qquad + \sum_{\substack{k\in \N \\ U_{k}\cap U_n\neq \varnothing }}\varphi_{k}(x)  \big(B(x)\uu^{k}(x)\big)_i \\
&\leqslant 
\sum_{\substack{k\in \N \\ U_{k}\cap U_n\neq \varnothing }}\varphi_{k}(x) \Big( H_i\big(x,Du^{k}_i(x)\big)+\big(B(x)\uu^{k}(x)\big)_i \Big) +\frac{\delta_n}{6}\\
&< -\frac{2}{3}\,\delta(x)+\frac{\delta_n}{6} \leqslant -\frac{\delta(x)}{2}. 
\end{align*}
This concludes the proof of the first part of the statement.

For the density, let $\uu$ be any critical subsolution. Let $\vv$ be a critical subsolution which is strict outside the Aubry set (whose existence is assured by Theorem \ref{teo Aubry}). Then, for any $\lambda\in (0,1)$, the function 
$(1-\lambda)\uu +\lambda \vv$ is a subsolution which is strict outside the Aubry set. This subsolution can therefore be regularized using the above procedure, giving a subsolution $\ww$ which is strict and smooth outside the Aubry set. Moreover, both these steps can be done in such a way that $\|\uu-\ww\|_\infty$ is as small as wanted. This establishes the density.\medskip
\end{dimo}

We now additionally assume the Hamiltonians $H_i$ to be strictly convex in $p$ and derive some further information on the behavior of Clarke's generalized gradients of the critical subsolutions on the Aubry set. 

We start with a preliminary lemma. 

\begin{lemma}
Let $y\in\A$ and let  $\uu^1, \cdots,\uu^\ell$ be critical subsolutions. Then, for all $i\in \ind$, 
$$\bigcap_{k=1}^{\ell} \partial^c u^k_i(x)\not=\varnothing.$$
Moreover, it contains a vector $p_i$ which is extremal for all the sets $\partial^c u^k_i(x)$ and which satisfies 
$$H_i(y,p_i)+\big(B(y)\uu^k(y)\big)_i=0\qquad\hbox{for every $k\in\{1,\dots,\ell\}$}.$$
\end{lemma}

\begin{dimo}
Let $\ww=\cfrac{1}{\ell}\sum\limits_{k=1}^\ell\uu^k\in \hh(0)$ and let  $p_i\in \partial^c w_i(y)$ be such that 
\[
H_i(y,p)+\big(B(y)\ww(y)\big)_i=0.
\] 
Such a $p_i$ must exist because otherwise $w_i$ would be strict at $y$. 
Note that, by strict convexity of $H_i$, the vector $p_i$ must be an extremal point of  $\partial^c w_i(x)$, hence it is a reachable gradient of $w_i$. Let $y_n\to y$ be such that  $u_i^k$ is differentiable at $y_n$ for every $k\in\{1,\dots,\ell\}$ and $n\in\N$, and 
$$Dw_i(y_n)=\frac{1}{\ell} \sum_{k=1}^\ell Du^k_i(y_n) \to p_i.$$
 Up to extraction of a subsequence, we can assume that $Du_i^k(y_n)\to q_k$ for all $k\in \{1,\dots,\ell\}$. Then one readily obtains, by  Jensen's inequality, that
 $$
0=H_i(y,p_i)+\big(B(y)\ww(y)\big)_i
\leqslant 
\frac{1}{\ell}\sum_{k=1}^\ell\Big( H_i(y,q_k)+\big(B(y)\uu^k(y)\big)_i\Big) 
\leqslant 0.
$$
Therefore, all the inequalities
$\frac{1}{\ell}\big(H_i(y,q_k)+(B(y)\uu^k(y))_i\big)\leqslant0$
summing to an equality, we deduce, by strict convexity of $H_i$, that $q_1=\cdots =q_l=p$. 
Moreover, since 
$$
H_i(y,q_k)+\big(B(y)\uu^k(y)\big)_i=0\qquad\hbox{for every $k\in\{1,\dots,\ell\}$},
$$ 
and because of the strict convexity of $H_i$, one sees that $p_i$ is extremal, and thus reachable, for all the $u_i^k$.\medskip
\end{dimo}

We now extend the previous result as follows: 

\begin{prop}\label{prop common gradient}
Let $y\in \A$.  Then, for each $i\in\ind$, there exists a vector $p_i\in \R^N$ which is a reachable gradient of 
$u_i$ at $y$ for every $\uu\in\hh(0)$ and which satisfies 
$$
H_i(y,p_i)+\big(B(y)\uu(y)\big)_i=0.
$$
\end{prop}

\begin{dimo}
For each critical subsolution $\uu$, let us denote by $P^{\uu}_i$ the set of reachable gradients  
$p$ of $u_i$ at $y$ that satisfy $H_i(y,p)+\big(B(y)\uu(y)\big)_i=0$. This set is not empty
 and compact. The proposition amounts to proving that 
\[
\bigcap_{\uu\in\hh(0)}P^{\uu}_i \neq \varnothing.
\]
If this were not the case, by compactness we could extract a finite empty intersection. But this would violate the previous lemma.\medskip
\end{dimo}

\end{section}

\begin{section}{Rigidity of the Aubry set and Comparison Principle}\label{comparison}

In this section we establish some interesting properties of the Aubry set and a comparison principle for the critical system. As we will see, these results follow rather easily thanks to the information gathered so far. 

We start with a remarkable rigidity phenomenon that takes place on the Aubry set.

\begin{teorema}\label{rigid}
Let $y\in\A$ and $i\in\ind$. Then
\[
 \vv(y)=\Phi_{\cdot,i}(y,y)+v_i(y)\1\qquad\hbox{for every $\vv\in\hh(0)$.}
\]
In particular, \quad $\vv(y)-\ww(y)\in\R\1$\quad for any $\vv,\,\ww\in\hh(0)$.
\end{teorema}

\begin{dimo}
Take $\vv\in\hh(0)$ and set $\uu:=\Phi_{\cdot,i}(y,\cdot)+v_i(y)\1$. According to Proposition \ref{Man sous}, $\uu$ is a critical solution satisfying $\vv\leqslant\uu$ in $\T^N$ and $v_i(y)=u_i(y)$. By applying Proposition \ref{prop pre-comparison} with $x_0:=y$ we get the assertion.\medskip
\end{dimo}

\begin{oss}\label{oss inferno}
On the other hand, the above property does not hold when $y\not\in\A$. Indeed, the proof of Lemma \ref{lemma inferno} shows that any 
critical subsolution $\vv$ which is strict at $y$ can be modified in such a way that the output is a critical subsolution all of whose components except one coincide at $y$ with those of $\vv$.  
\end{oss}
We derive two corollaries:
\begin{cor}
Let $y\in\A$. Then
the matrix $\Phi(y,y)$ is antisymetric.
\end{cor}
\begin{dimo}
Apply the previous theorem to the weak KAM solution $\Phi_{\cdot,j}(y,\cdot)$ and get
$$ \Phi_{\cdot,j}(y,y)=\Phi_{\cdot,i}(y,y)+\Phi_{i,j}(y,y)\1.$$
In particular,
$$ 0=\Phi_{j,j}(y,y)=\Phi_{j,i}(y,y)+\Phi_{i,j}(y,y).$$
\end{dimo}

\ 

\begin{cor}\label{cor2 rigid}
Let $y\in\A$. Then the critical solutions
 $\Phi_{\cdot,j}(y,\cdot)$ differ by a constant function.
 More precisely:
 $$\Phi_{\cdot,i}(y,\cdot)=\Phi_{\cdot,j}(y,\cdot)+\Phi_{j,i}(y,y)\1.$$
\end{cor}
\begin{dimo}
Let us apply the last point of Proposition \ref{Man sous} twice:
$$\Phi_{k,i}(y,z)\leqslant \Phi_{j,i}(y,y)+\Phi_{k,j}(y,z),$$
$$\Phi_{k,j}(y,z)\leqslant \Phi_{i,j}(y,y)+\Phi_{k,i}(y,z).$$
In particular, we obtain
$$\Phi_{k,i}(y,z)\leqslant \Phi_{j,i}(y,y)+\Phi_{k,j}(y,z)\leqslant \Phi_{j,i}(y,y)+\Phi_{i,j}(y,y)+\Phi_{k,i}(y,z)=\Phi_{k,i}(y,z),$$
thanks to the previous corollary. Therefore all inequalities are equalities and that gives the result.\end{dimo}

\medskip

Next, we derive a comparison principle for sub and supersolutions of the critical weakly coupled system \eqref{critical wcs c=0} which generalizes to our setting an analogous result established in \cite{leyetal} for Hamiltonians of a  special Eikonal form, see Subsection \ref{sez ley} for more details. In particular, we obtain that $\A$ is a uniqueness set for the critical system. 

\begin{teorema}\label{teo comparison}
Let $\mathbf{v},\,\mathbf{u}\in\big(\D{C}(\T^N)\big)^m$ be a sub and a supersolution of the critical weakly coupled system \eqref{critical wcs c=0}, respectively. Assume that  
\begin{equation}\label{hyp comparison}
\hbox{for every $x\in\A$\quad there exists $i\in\ind$ such that}\quad v_i(x)\leqslant u_i(x).     
\end{equation}
Then
\[
 \vv(x)\leqslant \uu(x)\qquad\hbox{for every $x\in\T^N$.}
\]
In particular, two critical solutions that coincide on $\A$ coincide on the whole $\T^N$. 
\end{teorema}

\begin{oss}\label{oss comparison}
The above theorem also implies that two critical solutions $\uu$ and $\vv$ are actually the same if \eqref{hyp comparison} holds with an equality. This is consistent with Theorem \ref{rigid}, which assures that this ``boundary'' condition amounts to requiring that $\uu=\vv$ on $\A$.  
\end{oss}

\begin{dimo}
In view of the density result stated in Theorem \ref{teo smooth subsol}, the critical subsolution $\vv$ can be approximated from below by a sequence of critical subsolutions that are, in addition, smooth and strict outside $\A$. Indeed, just pick a sequence $(\ww^n)_{n\in \mathbb \N}$ such that $\|\ww^n-\vv \|_\infty <n^{-1}$ and then define $\vv^n=\ww_n - n^{-1}\1$ which then verifies $\vv^n\leqslant \vv$ and $\|\vv^n-\vv \|_\infty <2n^{-1}$. Clearly, each element of the sequence still satisfies the boundary condition \eqref{hyp comparison}, hence it is enough to prove the statement by additionally assuming $\vv$ smooth and strict in $\T^N\setminus\A$. 

Let us set 
\[
M:=\max_{1\leqslant i\leqslant m}\max_{\T^N}\, (v_i-u_i)
\] 
and pick a point $x_0\in\T^N$ where such a maximum is attained. By Proposition \ref{prop pre-comparison} we know that $\vv(x_0)=\uu(x_0)+M\1$. 
If $x_0\not\in\A$, then $v_1$ would be a smooth subtangent to $u_1$ at $x_0$. The function $\uu$ being a supersolution, we would have
\[
0\leqslant H_1\big(x_0,Dv_1(x_0)\big)+\big(B(x_0)\uu(x)\big)_1=H_1\big(x_0,Dv_1(x_0)\big)+\big(B(x_0)\vv(x)\big)_1,
\]
in contrast with the fact that $\vv$ is strict in $\T^N\setminus\A$. Hence $x_0\in\A$ and by the hypothesis \eqref{hyp comparison} we get
$M\leqslant 0$, as it was to be proved.\medskip
\end{dimo}

Last, we show that the trace of any critical subsolution on the Aubry set can be extended to the whole torus in such a 
way that the output is a critical solution.

\begin{teorema}\label{teo Lax-type}
For any $\vv\in\hh(0)$, there exists a unique critical solution $\uu$ such that $\uu=\vv$ on $\A$.
\end{teorema}

\begin{dimo}
The assertion is derived by setting 
\[
 u_{i}(x)=\sup_{t>0} \big(\S(t)\vv\big)_{i}(x)\qquad\hbox{for every $x\in\T^N$ and $i\in\ind$.}
\]
Indeed, the functions $\{\S(t)\vv\,:\,t>0\}$ are equi--Lipschitz and non--decreasing with respect to $t$ and satisfy  $\S(t)\vv=\vv$ on $\A$ for every $t>0$ by Proposition \ref{prop S(t) su A}. We infer that $\uu$ is a vector valued, Lipschitz continuous function and $\S(t)\vv\ucv \uu$ in $\T^N$ as $t\to +\infty$. Last, $\uu$ is a critical solution for it is a fixed point of the semigroup $\S(t)$. 
\end{dimo}

\end{section}

\begin{section}{Examples}\label{examples}

The critical value and the Aubry set for a weakly coupled system of the kind studied in this paper have, in general, no connections with those of each Hamiltonian, considered individually. This happens also in simple situations, see Remark \ref{oss leyetal} below. 
In this section, we present some examples where more explicit results may be obtained for the critical value and for the Aubry set.

\subsection{The setting of \cite{leyetal}.}\label{sez ley}
The first example we propose corresponds to the setting considered in \cite{leyetal}. Assume that all the Hamiltonians are of the form $H_i(x,p)=F_i(x,p)-V_i(x)$, where:
\begin{itemize}
\item[(a)] $F_i$ and $V_i$ take non--negative values;\smallskip
\item[(b)] $F_i$ is convex and coercive in $p$;\smallskip
\item[(c)] $F_i(x,0)=0$ for all $x\in\T^N$ and $i\in \ind$.\smallskip
\end{itemize}
Furthermore, assume that 
\begin{equation}\label{condition leyetal}
\bigcap_{i=1}^m V_i^{-1}(\{0\}) \neq \varnothing.
\end{equation}

Under these hypotheses, we claim that the critical value is $0$ (whatever the coupling matrix is) and that the Aubry set is nothing but 
$$\A=\bigcap_{i=1}^m V_i^{-1}(\{0\}). $$
Indeed, it is easily seen that the  null function $\uu^0$ always belongs to $\hh(0)$ under the first set of hypotheses. Therefore, $\hh(0)\neq \varnothing$ and the critical value verifies $c\leqslant 0$. To see that there is actually equality, consider a point $x_0\in \cap V_i^{-1}(\{0\})$ and any ($ C^1$) function $\uu$. An easy consequence of Proposition \ref{prop Im} yields that $B(x_0)\uu(x_0)$ must have a non--negative entry, say $i$, hence 
$$
H_i\big(x_0,Du_i(x_0)\big) + \big(B(x_0)\uu(x_0)\big)_i=F_i\big(x_0,Du_i(x_0) \big)+\big(B(x_0)\uu(x_0)\big)_i\geqslant 0.
$$
Therefore, $\uu$ cannot belong to a $\hh(-\eps)$ for a positive $\eps$. The same argument can be adapted in the viscosity sense for any (non necessarily $C^1$) function. Therefore $0$ is the critical value.

To prove that $\cap V_i^{-1}(\{0\})$ is the Aubry set, first notice that, for every $y\not\in \cap  V_i^{-1}(\{0\})$, there exists an index $j$ such that $V_j(y)>0$. Then the $j$--th component of the null function $\uu^0$ is strict at $y$. In view of Theorem \ref{teo char A} we get the inclusion 
$$
\A\subseteq  \bigcap_{i=1}^m  V_i^{-1}(\{0\}).
$$
The opposite inclusion is obtained as previously. Take any $\uu\in \hh(0)$ and $x_0\in \cap V_i^{-1}(\{0\})$. We will do as if $\uu$ is differentiable at $x_0$, but the argument carries on in the general case using test functions and the viscosity subsolution property. At $x_0$ we must have 
$$
F_i\big(x_0,D u_i(x_0)\big)+\big(B(x_0) \uu(x_0)\big)_{i} \leqslant 0\qquad\hbox{for every $i\in\ind$}.
$$ 
But this is only possible if $F_i(x_0,Du_i(x_0))=0$ for all $i\in \ind$  and $B(x_0) \uu(x_0)=0$. Indeed, otherwise, $B(x_0) \uu(x_0)$ will have a positive entry in view of Proposition \ref{prop Im}, which is impossible. In particular, the above inequality holds with an equality. Since this happens for any critical subsolution $\uu$, we get $x_0\in \A$ in view of Theorem \ref{teo char A}. As a byproduct, this also establishes that at any point of $\A$, any critical subsolution must take as value a vector belonging to $\R\,\1$. This is a particular case of Theorem \ref{rigid} and accounts for the type of symmetries already remarked in \cite{leyetal} for the critical solutions obtained via the asymptotic procedure therein considered.

\begin{oss}\label{oss leyetal}
It would be interesting to understand what the Aubry set is for the weakly coupled system considered in the previous example when condition \eqref{condition leyetal} is dropped. Unfortunately, we are not able to give an answer to this question. Note that, in this case,  $c<0$. Indeed, if $c$ were greater or equal than $0$, then the null function would satisfy condition (iii) in Theorem \ref{teo char A} at any point $y\in\T^N$, contradicting the fact that the Aubry set is nonempty. We point out that  similar examples appear in \cite[Remark 3.5]{mitaketran} and \cite[Example 1.2]{Ng}.
\end{oss}

\subsection{The setting of \cite{leyetal} revisited.}\label{sez nguyen}
This second example is taken from \cite{Ng}: the Hamiltonians are still of the form $H_i(x,p)=F_i(x,p)-V_i(x)$, but the quantities $\lambda_i:=\min_{\T^N} V_i$ are not required to be zero. The analogous condition 
\begin{equation*}\label{condition nguyeb}
\bigcap_{i=1}^m V_i^{-1}(\{\lambda_i\}) \neq \varnothing
\end{equation*}
is in force. Moreover, the coupling matrix is taken independent of $x$. We claim that 
\[
c=-\pi({\boldsymbol
\lambda}) \qquad\hbox{and}\qquad \A=\bigcap_{i=1}^m V_i^{-1}(\{\lambda_i\}),
\]
where ${\boldsymbol \lambda}=(\lambda_1,\dots,\lambda_m)$ and $\pi(\llambda)$ denotes the unique real number such that 
$\llambda-\pi(\llambda)\1\in\D {Im}(B)$. Indeed, by replacing each $H_i$ with $H_i+\lambda_i$ we reduce to the case of Example \ref{sez ley} and we conclude by exploiting the following result:

\begin{prop}\label{prop new}
Let $H_i$ be convex and coercive Hamiltonians for every $i\in\ind$ and assume that the coupling matrix is independent of $x$.  For every ${\boldsymbol \lambda}=(\lambda_1,\dots,\lambda_m)$, denote by $c_{{\boldsymbol \lambda}}$ and $\A_{\boldsymbol \lambda}$ the critical value and Aubry set of the weakly coupled system with $H_i-{\lambda_i}$ in place of $H_i$ for every $i\in\ind$. Then 
\[
c_{\boldsymbol \lambda}=c_{\mathbf 0}-\pi({\boldsymbol \lambda})\qquad\hbox{and}\qquad \A_{\boldsymbol \lambda}=\A_{\mathbf 0} 
\qquad\hbox{for every ${\boldsymbol\lambda}\in\R^m$,}
\]
where $\pi(\llambda)$ denotes the unique real number such that 
$\llambda-\pi(\llambda)\1\in\D{Im}(B)$. 
\end{prop}

\begin{dimo}
Fix $\llambda=(\lambda_1,\dots,\lambda_m)\in\R^m$. Then $\llambda=\pi(\llambda)\1+B\mmu$ for some $\mmu\in\R^m$ and for a unique scalar $\pi(\llambda)$, for $\R^m\cong\D{Ker}(B)\oplus \D{Im}(B)$ in view of the results of Section \ref{sez linear algebra}.  
If $\uu$ is a solution of the critical weakly coupled system associated with $H_1,\dots,H_m$, then $\ww:=\uu+\mmu$ is a solution of 
\begin{equation}\label{eq new prop}
H_i(x,Dw_i)-\lambda_i+\big(B\ww(x)\big)_i=c_{\mathbf 0}-\pi(\llambda)\quad\hbox{in $\T^N$}\qquad\hbox{for every $i\in\ind$,}
\end{equation}
thus showing that $c_{\boldsymbol \lambda}=c_{\mathbf 0}-\pi({\boldsymbol \lambda})$. If we now take as $\uu$ a subsolution of the critical weakly coupled system associated with $H_1,\dots,H_m$ which is strict outside $\A_{\mathbf 0}$, we easily see that  
$\ww:=\uu+\mmu$ is a subsolution of \eqref{eq new prop} which is strict outside $\A_{\mathbf 0}$, thus showing 
$\A_{\llambda}\subseteq \A_{\mathbf 0}$. The reverse inclusion can be proved analogously. This concludes the proof. 
\end{dimo}

\subsection{Commuting Hamiltonians.}
In this last example we consider the case when the Hamiltonians are strictly convex and pairwise commute. If the Hamiltonians are of class $C^1$, that means
\begin{equation*}
 \{H_i,H_j\}(x,p):=\Big(\frac{\partial H_i}{\partial p}\frac{\partial H_j}{\partial x}-\frac{\partial H_j}{\partial p}\frac{\partial H_i}{\partial x}\Big)(x,p)=0\quad\hbox{in $\T^N\times \R^N$}
\end{equation*}
for every $i,j\in\ind$. If the Hamiltonians are only continuous, the commutation hypothesis must be expressed in terms of commutation of their Lax--Oleinik semigroup, see \cite{DavZav} for more details. We also make the additional assumption that, individually, all the Hamiltonians have $0$ as  critical value. Then, we claim that  $0$ is the critical value of the system as well (whatever the coupling is).

Indeed, it is proved in \cite{DavZav, zavido} that the Hamiltonians have the same critical solutions. In particular, there exists a function $u\in\D{Lip}(\T^N)$ satisfying
$$
H_i (x,D u)=0\quad\hbox{in $\T^N$}\qquad\hbox{for every $i\in\ind$}
$$
in the viscosity sense. 
Since the coupling is degenerate, we infer that the function 
$\uu^0 = u\1$ is a solution of
\begin{equation*}
H_i(x,Du^0_i)+\big(B(x)\uu^0(x)\big)_i=0\quad\hbox{in $\T^N$}\qquad\hbox{for every $i\in\ind$}.
\end{equation*}
Therefore, the claim is a direct consequence of Proposition \ref{prop unique c}.
Moreover, in this setting, we may localize the Aubry set of the system using those of the individual Hamiltonians. In order to do so, let us recall another result from \cite{DavZav}.
\begin{teorema}
Let $H_1,\cdots , H_m$ be pairwise commuting and strictly convex Hamiltonians, with common critical value equal to $0$. Then they have the same Aubry set $\A^*$. Moreover, there exists a common critical subsolution $v$ which is smooth outside $\A^*$ and strict for each Hamiltonian, i.e. 
$$
H_i\big(x,Dv(x)\big) < 0\quad\hbox{for every $x\in\T^N\setminus\A^*$ and $i\in\ind$}.
$$
\end{teorema}

Using this theorem, we easily see that the inclusion $\A\subseteq \A^*$ holds. Indeed, the function $\vv(x) := v(x)\1$  is a critical subsolution for the system which is strict outside $\A^*$. 

We also note that, as in the previous example, $\uu(y)\in\R\,\1$ for every $y\in\A$ and every $\uu\in\hh(0)$ in view of Theorem \ref{rigid}.

A particular case of this example is when all the $H_i$ are equal. In this case we get the more precise statement:
\begin{prop}
Let $H$ be a convex Hamiltonian and assume $H_1=\cdots = H_m=H$.  Then $\A = \A^*$. Moreover, all critical solutions of the system are of the form $\uu=u\1$ where $u$ is a critical solution of $H$.
\end{prop}
\begin{dimo}
The inclusion $\A\subseteq\A^*$ can be proved arguing as above (note that we do not need  the strict convexity assumption here). Let us prove the converse statement. Pick  
 $\vv \in \hh(0)$ and set $v(x) := \max_{i} v_i(x)$\quad for every $x\in\T^{N}$. We claim that $v$ is a critical subsolution for $H$.
Indeed, let $x\in \T^N$ and $p\in D^+v(x)$. Then $v(x)= v_i(x)$ for some $i\in\ind$. Since $v\geqslant v_i$ with equality at $x$, we get $p\in D^+v_i(x)$. We now use the fact that $\vv$ is a subsolution of the system to get 
\begin{equation}\label{right ineq}
H(x,p)\leqslant H(x,p)+\big(B(x)\vv(x)\big)_i\leqslant 0, 
\end{equation}
where the first inequality comes from the fact that
$$
\big(B(x)\vv(x)\big)_i=\sum_{j=1}^m b_{ij}(x)v_j(x)\geqslant \sum_{j=1}^m b_{ij}(x)v(x)=0,
$$
which holds true since $b_{ij}(x)\leqslant 0$ and $v_j(x)\leqslant v(x)$ for every $j \neq i$.
Let us now assume that $\vv$ is strict outside $\A$. Then the right inequality in \eqref{right ineq} is strict as soon as $x\notin \A$, yielding that $v$ is a subsolution for $H$ which strict in the complementary of $\A$. This proves that $\A^*\subseteq \A$, hence  $\A=\A^*$.

Let now $\uu$ be a critical solution for the system. Then $v(x) := \max_{i} v_i(x)$ is a critical subsolution for $H$. Moreover, as
$$ 
\uu(x)=u_1(x)\1\quad\hbox{for every $x\in\A$},
$$ 
we deduce that $v=u_1$ on $\A$. Since $\A=\A^*$, there exists a critical solution $\tilde u$ for $H$ such that $\tilde u=v$ on $\A$. Now the function $\tilde  \uu = \tilde u \1$ is a critical solution of the weakly coupled system satisfying $\tilde\uu=\uu$ on $\A$. By the comparison principle, i.e. Theorem \ref{teo comparison}, we conclude that $\uu = \tilde \uu$.
\end{dimo}

\end{section}

\begin{appendix}
\begin{section}{}\label{appendix semiconcave}

In this appendix we want to give a proof of Proposition \ref{prop invariant}.\smallskip

\noindent In what follows, a function $u$ will be said to be {\em semiconcave} on an open subset $U$ of either $\T^N$ or $\R_{+}\times\T^N$ if, for every $x\in U$, there exists  a vector $p_x\in\R^N$ such that 
\begin{equation*}\label{cond semiconcavita}
 u(y)-u(x)\leqslant \langle p_x,y-x\rangle + d(y,x)\,\omega\big(d(y,x)\big)\qquad\hbox{for every $y\in U$,}
\end{equation*}
where $\omega$ is a modulus. It can be shown this is equivalent to requiring that for every $x,y \in U$ and $\lambda \in [0,1]$,
$$\lambda u(x) + (1-\lambda) u(y) \leqslant u\big( \lambda x + (1-\lambda)y \big) + \lambda(1-\lambda) \omega\big(d(x,y)\big).$$
 The vectors $p_x$ satisfying the above inequality are precisely the elements of $D^+ u(x)$, which is thus always nonempty in $U$. Moreover, $\partial^c u(x)=D^+u(x)$ for every $x\in U$. By the upper semicontinuity of the map $x\mapsto\partial^c u(x)$ with respect to set inclusion, we get in particular that $Du$ is continuous in its domain of definition, see \cite{CaSi00}.\smallskip

We start with the following

\begin{prop}\label{finoa0}
Let $T>0$ and $G :[0,T] \times \T^N \times \R^N\to \R$ be a locally Lipschitz Hamiltonian such that $G(s,\cdot,\cdot)$ is a strictly convex Hamiltonian, for every fixed $s\in [0,T]$. Let $u(t,x)$ be a Lipschitz function in $[0,T]\times\T^N$ that solves the evolutive Hamilton-Jacobi equation
\begin{equation}\label{eq G}
\frac{\partial u}{\partial t} + G(t,x,D_x u) = 0\qquad\hbox{in $(0,T)\times\T^N$},
\end{equation}
in the viscosity sense. Then
\begin{itemize}
\item[{\em (i)}] for every $0<\tau<T$, the function $u$ is semiconcave in $[\tau,T)\times \T^N$;\smallskip 
 \item[{\em (ii)}] if $u (0,\cdot)$ is semi--concave in $\T^N$, then  the functions $\{\,u(t, \cdot)\,:\, t\in [0,T)\,\}$ are equi--semiconcave.
\end{itemize}
\end{prop}

\begin{dimo}
Since $u$ is Lipschitz, up to modifying $G$ outside $[0,T]\times\T^N\times B_R$ for a sufficiently large $R>0$, we can assume that $G$ is superlinear in $p$, uniformly with respect to $(t,x)$. We are then in the setting considered by 
Cannarsa and Soner in \cite{canson} and item {\em (i)} follows from their results.

Let us prove {\em (ii)}. 
Let us denote by $L(t,x,q)$ the the Lagrangian associated with $G$ through the Fenchel transform  and by $u_0$ the initial datum $u (0,\cdot)$. It is well known, see for instance \cite{CaSi00}, that the following representation formula holds:
\begin{equation}\label{Lax-Oleinik formula}
u(t,x)=\inf_{\xi(t)=x}\Big( u_0\big(\xi(0)\big)+\int_0^t L\big(s,\xi (s), \dot \xi (s) \big) \dd s\Big),\qquad (t,x)\in (0,T)\times\T^N,
\end{equation}
where the infimum is taken by letting $\xi$ vary in the family of absolutely continuous curves from $[0,t]$ to $\T^N$. Moreover, the
minimum is attained by some curve $\gamma$, which is, in addition, Lipschitz continuous (actually, of class $ C^1$), see \cite{ClVi85}. 

We claim that there exists a constant $\kappa$, only depending on $G$ and on the Lipschitz constant of $u$ in $[0,T]\times\T^N$, such that $\|\dot\gamma\|_\infty\leqslant \kappa$. To this aim, we apply  Proposition 2.4 in \cite{Ishii08} to the function $u(t,x)$ and the curve $s\mapsto (s,\gamma(s))$  to get 
\begin{equation}\label{eq derivata u}
\frac{\dd}{\dd s}\,u\big(s,\gamma(s)\big)=p_t(s)+\langle p_x(s),\dot\gamma(s)\rangle \qquad \hbox{for a.e. $s\in [0,t]$},
\end{equation}
where $s\mapsto \big(p_t(s),p_x(s)\big)$ is a measurable and essentially bounded function on $[0,t]$ such that 
\[
 \big(p_t(s),p_x(s)\big)\in\partial^c u\big(s,\gamma(s)\big) \qquad\hbox{for a.e. $s\in [0,t]$}.
\]
By integrating \eqref{eq derivata u} and using the Fenchel inequality we get
\begin{eqnarray*}
u(t,x)&=&u_0\big(\gamma(0)\big)+\int_0^t p_t(s)+\langle p_x(s),\dot\gamma(s)\rangle \,\dd s\\
     &\leqslant&
  u_0\big(\gamma(0)\big)+\int_0^t p_t(s)+G\big(s,\gamma(s),p_x(s)\big)+L\big(s,\gamma(s),\dot\gamma(s)\big) \,\dd s\\
     &\leqslant& 
  u_0\big(\gamma(0)\big)+\int_0^t L\big(s,\gamma(s),\dot\gamma(s)\big) \,\dd s,
\end{eqnarray*}
where in the last inequality we used the fact that $u$ is a (sub)--solution of the time dependent equation, i.e. 
\[
 p_t+G(t,x,p_x)\leqslant 0\qquad\hbox{for every $(p_t,p_x)\in\partial^c u(t,x)$ and $(t,x)\in (0,T)\times\T^N$.}
\]
Since $\gamma$ is minimizing, all the inequalities must be equalities, in particular we obtain
\begin{equation}
\dot\gamma(s)\in\partial_p G\big(s,\gamma(s),p_x(s)\big)\qquad\hbox{for a.e. $s\in [0,t]$}.
\end{equation}
This proves the claim by choosing 
\[
 \kappa:=\sup\big\{|q|\,:\,q\in\partial_p G(s,x,p),\,(s,x)\in [0,T]\times\T^N,\,|p|\leqslant \D{Lip}\left(u;[0,T]\times\T^N\right)\,\big\},
\]
which is finite thanks to the convexity and the growth assumptions assumed on $G$ with respect to $p$. 

Let us now fix $t\in (0,T)$, $x_1, x_2\in \T^N$,  $\lambda \in [0,1]$ and set $x=\lambda x_1+ (1-\lambda) x_2$. Note that 
$x_1 = x+ (1-\lambda)h$ and $x_2= x- \lambda h$ for  $h=x_1-x_2$. Let us denote by $\gamma$ a curve realizing the infimum in \eqref{Lax-Oleinik formula} for such a pair of $(t,x)$, by $K$ a Lipschitz constant for $L$ restricted to $ [0,T] \times \T^N \times B(0,2\kappa)$ and by $\omega$ a semi--concavity modulus for $u_0$. We get 
\begin{align*}
\lambda u(t,x_1)&+ (1-\lambda) u(t,x_2) - u(t,x)\\
&\leqslant \lambda \Big( u_0\big(\gamma(0)+(1-\lambda )h\big)
+\int_0^t L\big(s,\gamma (s)+(1-\lambda ) h , \dot \gamma (s) \big) \dd s \Big)\\
&\qquad+  (1-\lambda ) \Big( u_0\big(\gamma(0)-\lambda h\big)+\int_0^t L\big(s,\gamma (s)-\lambda  h , \dot \gamma (s) \big) \dd s \Big) \\
&\qquad- \Big( u_0\big(\gamma(0)\big)+\int_0^t L\big(s,\gamma (s), \dot \gamma (s) \big)\dd s\Big)\\
&= \lambda u_0\big(\gamma(0)+(1-\lambda )h\big) + (1-\lambda) u_0\big(\gamma(0)-\lambda h\big) - u_0\big(\gamma(0)\big)\\
&\qquad+ \lambda \Big(\int_0^t \Big(L\big(s,\gamma (s)+(1-\lambda ) h , \dot \gamma (s) \big) -  L\big(s,\gamma (s), \dot \gamma (s) \big)\Big)\dd s \\
&\qquad+  (1-\lambda ) \int_0^t \Big( L\big(s,\gamma (s)-\lambda  h , \dot \gamma (s) \big) -  L\big(s,\gamma (s), \dot \gamma (s) \big)\Big) \dd s \\
&\leqslant \lambda (1-\lambda)\Big( \omega ( d(x_1,x_2)) + t\, Kd(x_1,x_2)\Big),
\end{align*}
which proves the assertion.\medskip
\end{dimo}

\medskip

The result just proved will be applied to weakly coupled systems as follows: 

\begin{prop}\label{regularity}
Let $T>0$ and $\uu=(u_1,\dots,u_m)\in \big(\D{Lip}([0,T]\times\T^N)\big)^m$ be a solution of the evolutionary weakly coupled system (\ref{evo wcoupled system}). 
Let $B(x)$ be Lipschitz and $H_i$ be
 {locally Lipschitz and strictly convex}, for some fixed index 
$i\in\ind$. Then, for all $0<\tau <T$,  the function 
 $u_i$ restricted to $[\tau,T)\times \T^N$ is semiconcave. 
 Moreover, if, the initial condition $u_i(0,\cdot)$ is semiconcave, then  the functions $\{\,u_i(t, \cdot)\,:\, t\in [0,T]\,\}$ are equi--semiconcave.
\end{prop}

\begin{proof}
The function $u_i$ solves, for the given index $i\in \ind$,  a Hamilton--Jacobi equation of the kind \eqref{eq G} with 
\[
 G(t,x,p)=H_i(x,p)+\big(B(x)\uu(t,x)\big)_i,\qquad (t,x,p)\in [0,T]\times\T^N\times\R^N.
\]
The conclusion follows by applying Proposition \ref{finoa0}.
\end{proof}

We are now ready to prove Proposition \ref{prop invariant}.\\

\noindent{\bf Proof of Proposition \ref{prop invariant}.}
We recall that, by convexity of the Hamiltonians, subsolutions to the critical system coincide with almost everywhere subsolutions. This fact will be repeatedly exploited along the proof. 

Assume first that $t\mapsto \S(t)\uu+t\,a\1$ is non--decreasing. Pick $t_0>0$ such that the map $(t,x)\mapsto \S(t)\uu(x)$ is differentiable at $(t_0,x)$ for almost every $x\in\T^N$ and 
\[
  \partial_t \S(t_0)\uu(x) \geqslant -a\1\qquad\hbox{for a.e. $x\in\T^N$.}
\]
By the Lipschitz character of the map $(t,x)\mapsto \S(t)\uu(x)$ and Fubini's theorem, this holds true for almost every $t_0>0$.
Using the evolutionary equation, which is verified at every differentiability point of $ \S(t)\uu(x)$, we deduce that, for every $i\in\ind$,
$$
H_i\Big(x,D\big(\S(t_0)\uu\big)_i(x)\Big)+\big(B(x)\S(t_0)\uu(x)\big)_i\leqslant a \quad\hbox{for a.e. $x\in\T^N$},
$$
that is, $\S(t_0)\uu\in \hh(a)$. This being true for almost every $t_0>0$, the conclusion follows by stability of viscosity subsolutions.

Let us now assume reciprocally that $\uu \in\hh (a)$. We first approximate each Hamiltonian $H_i$ with a sequence $(H_i^k)_k$ of convex Hamiltonians that are, in addition, locally Lipschitz in $(x,p)$ and strictly convex in $p$. This can be done by taking a sequence $(\rho_k)_k$ of standard mollifiers on $\R^N$ and by setting
\[
 H_i^k(x,p)=\int_{B_1} \rho_k(y)H_i(x-y,p)\,\dd y +\frac{|p|^2}{k},\qquad (x,p)\in\T^N\times\R^N.
\]
Analogously, we approximate the matrix $B(x)$ by a sequence of coupling matrixes $\big(B_k(x)\big)_k$ that are Lipschitz in $x$. 
Note that, for each index $i\in\ind$, \quad $H_i^k\ucv H_i$ \quad in $\T^N\times\R^N$ and $B_k\ucv B$ in $\T^N$ as $k\to +\infty$. Let us denote by $\hh_k(a)$ the set of $\aaa$--subsolution of the weakly coupled system \eqref{wcoupled system} 
with $\aaa=a\1$ and with $H_1^k,\dots,H_m^k$ and $B_k$ in place of $H_1,\dots,H_m$ and $B$, respectively, and by $\S_k$ the semigroup associated with the corresponding time--dependent equation \eqref{evo wcoupled system}.   

Next, we approximate $\uu$ with a sequence of $(\uu^n)_n$ of functions that are component--wise semi--concave by setting
$$
u^n_i(x) = \inf_{y\in \T^N } u_i(y)+nd(y,x)^2 \qquad\hbox{for every $x\in\T^N$ and $i=1,\cdots ,m$}.
$$
Fix $\eps>0$. A standard argument shows that, for $n$ large enough, $\uu^n\in \hh (a+\eps)$. Moreover, by the Lipschitz character of $\uu^n$ and by the local uniform convergence of $(H_1^k,\dots,H_m^k)$ to $(H_1,\dots,H_m)$ and of $B_k$ to $B$, we also have that $\uu^n\in \hh_k (a+2\eps)$ for $k$ sufficiently large. We now apply Proposition \ref{regularity} to infer that  the map $(t,x)\mapsto\S_k(t)\uu^n(x)$ is semiconcave in $[0,\tau]\times\T^N$ for every $\tau>0$. By using the  fact that the gradient of a semiconcave function is continuous in its domain of definition and by choosing  $\tau>0$ small enough, we get  
$\S_k(t)\uu^n\in \hh_k (a+3\eps)$ for every $0\leqslant t \leqslant \tau$. 
By exploiting this information in the evolutive weakly coupled system, we get 
\[
\frac{\partial}{\partial t}\S_k(t)\uu^n(x) \geqslant -(a+3\eps)\1\qquad\hbox{for a.e. $(t,x)\in (0,\tau)\times\T^N$,}
\]
i.e.
\[
\S_k(t+h)\uu^n \geqslant \S_k(t)\uu^n -h (a+3\eps)\1\qquad\hbox{for every $0<t<t+h\leqslant\tau$}.
\]
Now, by the comparison principle for the evolution equation and by using the fact that the semigroup commutes with the addition of scalar multiples of the vector $\1$, we obtain that $t\mapsto \S_k(t)\uu^n -t(a+3\eps)\1$ is non decreasing. 
We now exploit the fact that 
\[
 \S_k(t)\uu^n\underset{k\to +\infty}{\ucv}\S(t)\uu^n\quad\hbox{and}\quad \S(t)\uu^n\underset{n\to +\infty}{\ucv}\S(t)\uu\qquad\hbox{in $\R_+\times\T^N$}
\]
to infer that 
\[
 t\mapsto \S(t)\uu^n -t(a+3\eps)\1\quad\hbox{is non--decreasing on $[0,+\infty)$.}
\]
Being this true for every $\eps>0$, we finally have that $t\mapsto \S(t)\uu^n -ta\1$ is non--decreasing on $[0,+\infty)$.

The last assertion follows from the equivalence just proved, together with the fact that the semigroup $\S(t)$ is non--decreasing and commutes with addition of vectors of the form $a\,\1$ with $a\in\R$.\qed\medskip

%
%

\end{section}
\end{appendix}

\bibliography{weakly}
\bibliographystyle{siam}

\end{document}